\documentclass[11pt,reqno]{amsart}

\usepackage[T1]{fontenc}
\usepackage{lmodern}
\usepackage{amsmath,amssymb,mathtools,mathrsfs}
\usepackage{booktabs}
\usepackage{microtype}
\usepackage{enumitem}
\usepackage{listings}
\usepackage{url}
\usepackage[hidelinks]{hyperref}

\hypersetup{
  pdftitle={Derivative Sums of Balanced Gamma Quotients and Multiple Zeta Values: Five Conjectures of Zhi-Wei Sun},
  pdfauthor={Shivam Nalin Patel},
  pdfsubject={Balanced gamma quotients, Wilf--Zeilberger identities, and multiple zeta values},
  pdfkeywords={balanced gamma quotient, derivative sum, inverse central binomial sum, Wilf--Zeilberger identity, multiple zeta value, exact certificate}
}

\allowdisplaybreaks
\numberwithin{equation}{section}

\newtheorem{theorem}{Theorem}[section]
\newtheorem{proposition}[theorem]{Proposition}
\newtheorem{lemma}[theorem]{Lemma}
\newtheorem{corollary}[theorem]{Corollary}
\theoremstyle{definition}
\newtheorem{definition}[theorem]{Definition}
\theoremstyle{remark}
\newtheorem{remark}[theorem]{Remark}
\newtheorem{conjecture}[theorem]{Conjecture}

\newcommand{\Coeff}{\mathcal C}
\newcommand{\opA}{\mathsf A}
\newcommand{\opB}{\mathsf B}
\newcommand{\opC}{\mathsf C}
\newcommand{\opD}{\mathsf D}
\newcommand{\Q}{\mathbb Q}

\newcommand{\onevec}[1]{\{1\}^{#1}}

\newcommand{\Span}{\operatorname{span}_{\Q}}

\newcommand{\RF}[2]{(#1)_{#2}}

\lstdefinestyle{proofcode}{
  language=Python,
  basicstyle=\ttfamily\tiny,
  numbers=left,
  numberstyle=\tiny,
  stepnumber=1,
  numbersep=6pt,
  frame=single,
  breaklines=true,
  breakatwhitespace=false,
  columns=fullflexible,
  keepspaces=true,
  showstringspaces=false,
  tabsize=4,
  captionpos=b
}

\title[Derivative sums and five conjectures]
{Derivative Sums of Balanced Gamma Quotients and Multiple Zeta Values:\\
Five Conjectures of Zhi-Wei Sun}
\author{Shivam Nalin Patel}
\address{Independent Researcher, Newark, California, USA}
\email{patelshivam99@gmail.com}
\date{}


\begin{document}
\raggedbottom

\begin{abstract}
We introduce a uniform reduction for derivative sums of balanced gamma
quotients.  For exponent data $(a_i,e_i)$ satisfying
$\sum_i e_i a_i=0$, the translation-dependent gamma prefactor is governed
by the characteristic power sums $\chi_m=\sum_i e_i a_i^m$ through
\[
 \log C(u)=\sum_{m\ge2}(-1)^m\frac{\chi_m}{m}\zeta(m)u^m.
\]
This separates the universal gamma contribution from a hypergeometric
coefficient-extraction problem and organizes three exponential families.
For $\chi_m=2-2^m$, diagonal and symmetric specializations of a
four-parameter Wilf--Zeilberger identity prove Conjectures~4.2 and~4.3 of
Zhi-Wei Sun.  For $\chi_m=3-3^m$, exact span certificates in the
coefficient spaces of Au's Example~IV prove corrected forms of
Conjectures~4.4 and~4.5.  For
$\chi_m=4^m-10\cdot2^m+16$, a half-integer specialization of Au's first
$1/\pi^4$ construction, combined with the diagonal transformation in his
Example~VI, proves Conjecture~4.6 through weight eleven.  The transformed
sides reduce to ordinary multiple zeta values, and every computer-assisted
acceptance test is exact: explicit rational WZ certificates and separately
implemented MZV certificate checkers use no numerical recognition, PSLQ,
or conjectural MZV dimensions.  We also identify four errors in the printed
statements of Conjectures~4.2--4.5.
\end{abstract}

\maketitle

\section{Introduction}

Let $\psi_1=\psi'$ be the trigamma function and define, for $x>0$,
\begin{equation}\label{eq:f123}
\begin{aligned}
 f_1(x)&=\frac{\Gamma(x)^2}{2\Gamma(2x)},\\
 f_2(x)&=\bigl(\zeta(2)-\psi_1(x)\bigr)
        \frac{\Gamma(x)^2}{2x\Gamma(2x)},\\
 f_3(x)&=\frac{\Gamma(x)^2}{2x^3\Gamma(2x)}.
\end{aligned}
\end{equation}
For the cubic branch put
\begin{align}
 \Phi(x)&=e^{\pi i x}
 \frac{(28x^2-18x+3)\Gamma(x)^3}
 {3x^2(2x-1)^3\Gamma(3x)},\label{eq:Fdef}\\
 \Psi(x)&=e^{\pi i x}
 \frac{P(x)\Gamma(x)^3}
 {3x^3(2x-1)^4\Gamma(3x)},\label{eq:Psidef}\\
 P(x)&=560x^4-640x^3+408x^2-136x+17.\label{eq:Pdef}
\end{align}
For the inverse-$4096$ branch put
\begin{align}
 P_6(x)&=5460x^4-8341x^3+4864x^2-1280x+128,\label{eq:intro-P6}\\
 f_6(x)&=\frac{P_6(x)\Gamma(4x-1)\Gamma(x)^{16}}
 {256\Gamma(2x)^{10}}.\label{eq:intro-f6}
\end{align}
Theorem~\ref{thm:f6} proves Sun's Conjecture~4.6 for this function.

The functions in \eqref{eq:f123} and \eqref{eq:Fdef}--\eqref{eq:Pdef} are those labelled $f_2,f_3,f_4,f_5$ in Conjectures~4.2--4.5 of Sun's collection \cite{Sun2026}; \eqref{eq:intro-f6} is Sun's $f_6$.  At positive
integers,
\begin{equation}\label{eq:integer-values}
 f_2(n)=\frac{H_{n-1}^{(2)}}{n^2\binom{2n}{n}},\qquad
 f_3(n)=\frac{1}{n^4\binom{2n}{n}}.
\end{equation}
The companion paper \cite{Patel2026} treats Conjecture~4.1 by a different,
beta-integral method.  No result from that paper is used here.

We use the convention
\begin{equation}\label{eq:mzv-def}
 \zeta(s_1,\ldots,s_d)=
 \sum_{n_1>\cdots>n_d\ge1}\frac{1}{n_1^{s_1}\cdots n_d^{s_d}},
 \qquad s_1\ge2.
\end{equation}

\begin{theorem}\label{thm:main}
For the functions in \eqref{eq:f123},
\begin{align}
 \sum_{n=1}^{\infty}f_2'(n)
 &=\frac{29}{27}\zeta(5), \label{eq:f2-1}\\
 \sum_{n=1}^{\infty}f_2''(n)
 &=-\frac{4115}{648}\zeta(6)-\frac83\zeta(3)^2, \label{eq:f2-2}\\
 \sum_{n=1}^{\infty}f_2^{(3)}(n)
 &=\frac{432\zeta(3)\zeta(4)+1069\zeta(7)-116\zeta(2)\zeta(5)}{18}, \label{eq:f2-3}\\
 \sum_{n=1}^{\infty}f_2^{(4)}(n)
 &=-\frac{359939}{648}\zeta(8)-128\zeta(3)\zeta(5)
 +32\zeta(2)\zeta(3)^2+\frac{224}{9}\zeta(5,3), \label{eq:f2-4}\\
 \sum_{n=1}^{\infty}f_2^{(5)}(n)
 &=\frac{459385}{81}\zeta(9)-\frac{10690}{9}\zeta(2)\zeta(7)
 -\frac{640}{3}\zeta(3)^3 \notag\\
 &\hspace{18mm}+1990\zeta(4)\zeta(5)-40\zeta(3)\zeta(6), \label{eq:f2-5}
\end{align}
and
\begin{align}
 \sum_{n=1}^{\infty}f_3''(n)&=\frac{3073}{216}\zeta(6), \label{eq:f3-2}\\
 \sum_{n=1}^{\infty}f_3^{(3)}(n)&=\frac{176\zeta(2)\zeta(5)-1423\zeta(7)}{12}, \label{eq:f3-3}\\
 \sum_{n=1}^{\infty}f_3^{(4)}(n)&=\frac{752537\zeta(8)+25344\zeta(5,3)}{1080}, \label{eq:f3-4}\\
 \sum_{n=1}^{\infty}f_3^{(5)}(n)&=660\zeta(4)\zeta(5)+\frac{7115}{3}\zeta(2)\zeta(7)
 -\frac{283010}{27}\zeta(9). \label{eq:f3-5}
\end{align}
Consequently Sun's Conjectures~4.2 and~4.3 hold.
\end{theorem}

Remove the phases from \eqref{eq:Fdef}--\eqref{eq:Psidef} by writing
\begin{equation}\label{eq:hdefs}
 U(x)=e^{-\pi i x}\Phi(x),\qquad V(x)=e^{-\pi i x}\Psi(x),
\end{equation}
and set
\begin{equation}\label{eq:Rdefs}
 R_r^U=\sum_{n=1}^{\infty}(-1)^nU^{(r)}(n),\qquad
 R_r^V=\sum_{n=1}^{\infty}(-1)^nV^{(r)}(n).
\end{equation}

\begin{theorem}\label{thm:real}
The alternating real sums for $U$ are
\begin{align}
 R_0^U&=-\frac{\pi^4}{45},& R_1^U&=20\zeta(5),\label{eq:RA01}\\
 R_2^U&=-219\zeta(6),\label{eq:RA2}\\
 R_3^U&=3276\zeta(7)-360\zeta(2)\zeta(5),\label{eq:RA3}\\
 R_4^U&=-36258\zeta(8)-768\zeta(5,3),\label{eq:RA4}\\
 R_5^U&=898560\zeta(9)-196560\zeta(2)\zeta(7)-19800\zeta(4)\zeta(5).\label{eq:RA5}
\end{align}
The corresponding sums for $V$ are
\begin{align}
 R_0^V&=180\zeta(5)-112\zeta(2)\zeta(3),\label{eq:RB0}\\
 R_1^V&=392\zeta(3)^2-190\zeta(6),\label{eq:RB1}\\
 R_2^V&=3840\zeta(7)-1080\zeta(2)\zeta(5)-5040\zeta(3)\zeta(4),\label{eq:RB2}\\
 R_3^V&=\frac15\bigl(2298\zeta(8)+312480\zeta(3)\zeta(5)+53856\zeta(5,3)
 -35280\zeta(2)\zeta(3)^2\bigr),\label{eq:RB3}\\
 R_4^V&=25088\zeta(3)^3-400680\zeta(3)\zeta(6)
 \notag\\
 &\quad-35640\zeta(4)\zeta(5)-138240\zeta(2)\zeta(7)
 -236480\zeta(9).\label{eq:RB4}
\end{align}
Together with the elementary phase-removal formula, these give the corrected
forms of Sun's Conjectures~4.4 and~4.5.
\end{theorem}

\begin{proposition}[Corrections to the source statements]\label{prop:errata}
The following four changes are required in \cite[Section~4]{Sun2026}.
\begin{enumerate}[label=\textup{(\roman*)}]
\item In (4.12), the left side must be $\sum_{k\ge1}f_3^{(5)}(k)$, not
$\sum_{k\ge1}f_4^{(5)}(k)$.
\item In (4.14), the right side must be
$-198\zeta(6)+40\pi i\zeta(5)$.
\item In (4.17), the imaginary term must be
$-240\pi i\bigl(665\zeta(8)+16\zeta(5,3)\bigr)$.
\item In Conjecture~4.5, the coefficient of $x^2$ in Sun's polynomial
$S(x)$ must be $408$, not $480$.
\end{enumerate}
\end{proposition}
The first item is forced by the surrounding definition of $f_3$.  The next
two follow already from Leibniz's rule after phase removal.  The final one is
forced by the proved zeroth-order identity in Au's Example~IV and the
identity discussed in \cite{SunMO2025,Au2026}; the polynomial with $480$
does not lie in the required WZ coefficient space.

\subsection*{Relation to earlier work and contribution}
Classical inverse central-binomial series go back at least to Lehmer and
Zucker, with systematic half-integer and all-order treatments developed by
Weinzierl and by Kalmykov, Ward, and Yost
\cite{Lehmer1985,Zucker1985,Weinzierl2004,KWY2007}.  Iterated-integral,
colored-MZV, and cyclotomic approaches include
\cite{Ablinger2014,CGZ2024,XuZhao2023,Zhou2023,SunZhou2024,SunZhou2026}.
Inverse central-binomial sums naturally connect to sixth-root
polylogarithms and multiple Clausen values \cite{BBK2001}, while WZ seeds
provide exact transformations to coefficient spaces whose simple sides are
multiple zeta values \cite{Au2023,AuSeeds2025,Au2026}.  Hou and Sun use
related parameter differentiation to prove other derivative identities
\cite{HouSun2026}; their target gamma quotients are different.  The new
structural ingredient here is the balanced-quotient reduction of
Section~\ref{sec:balanced}.  The branch-specific ingredients are the
diagonal and symmetric slices $(t,t,2t,t)$ and $(t+s,t-s,2t,t)$, the exact
functionals $\Delta_r$, and the joint use of all three coefficient families
in Au's Example~IV.  Thus the paper is not a concatenation of isolated evaluations: it combines one
balanced-quotient reduction, three branch-specific WZ constructions,
twenty-eight exact derivative or zeroth-order evaluations, and four source
corrections.

\section{Why a direct beta-integral treatment does not organize these cases}
A beta-integral or log-sine expansion is effective for low-weight inverse
central-binomial sums and is the method used for the companion $f_1$ case
\cite{Patel2026}.  It does not, however, place the present identities
manifestly in the ordinary MZV algebra.  Even elementary inverse-central-
binomial sums naturally produce sixth-root periods; for example the
systematic framework of Borwein, Broadhurst, and Kamnitzer involves
multiple Clausen values and polylogarithms at sixth roots of unity
\cite{BBK2001}.  At weights eight and nine, a direct integral expansion
would therefore create a large cyclotomic expression and leave the decisive
cancellation to ordinary MZVs invisible.

The WZ transformations used below build that descent into the identity.
After coefficient extraction, their simple sides are finite ordinary-MZV
combinations of controlled weight.  Exact rational certificates then prove
the reductions.  Numerical recognition is not an alternative proof and is
particularly unsafe here: three substantive coefficients and one function
label in the original conjecture list require correction.

\section{Balanced gamma quotients}\label{sec:balanced}

\begin{definition}
Let $a_i\in\mathbb Z_{>0}$ and $e_i\in\mathbb Z$.  A gamma quotient
\[
 \mathcal G(x)=R(x)\prod_{i=1}^s\Gamma(a_i x)^{e_i},
\]
with $R$ rational, is \emph{balanced} if
\[
 \chi_1:=\sum_{i=1}^s e_i a_i=0.
\]
Its characteristic power sums are
\begin{equation}\label{eq:chi-def}
 \chi_m=\sum_{i=1}^s e_i a_i^m\qquad(m\ge1).
\end{equation}
\end{definition}

\begin{lemma}[Universal translation prefactor]\label{lem:balanced-prefactor}
For every positive integer $n$ and $u$ near zero,
\begin{equation}\label{eq:balanced-factor}
 \mathcal G(n+u)=C_{\mathcal G}(u)K_{\mathcal G}(n,u),
\end{equation}
where
\begin{align}
 C_{\mathcal G}(u)&=\prod_{i=1}^s\Gamma(1+a_i u)^{e_i},\label{eq:CG}\\
 K_{\mathcal G}(n,u)&=R(n+u)
 \prod_{i=1}^s\left(
 \frac{(1+a_i u)_{a_i n}}{a_i(n+u)}\right)^{e_i}.\label{eq:KG}
\end{align}
If the quotient is balanced, then
\begin{equation}\label{eq:log-general-C}
 \boxed{\ \log C_{\mathcal G}(u)=
 \sum_{m=2}^{\infty}(-1)^m\frac{\chi_m}{m}\zeta(m)u^m\ }.
\end{equation}
\end{lemma}

\begin{proof}
The identity
$\Gamma(a_i(n+u))=\Gamma(1+a_i u)(1+a_i u)_{a_i n}/(a_i(n+u))$
gives \eqref{eq:balanced-factor}.  Now use
\[
 \log\Gamma(1+z)=-\gamma z+
 \sum_{m\ge2}(-1)^m\frac{\zeta(m)}m z^m.
\]
The Euler--Mascheroni term has coefficient $\chi_1$ and vanishes exactly
under the balance condition.
\end{proof}

\begin{proposition}[Exponential base and normal convergence]\label{prop:base}
The exponential base of the pure gamma quotient is
\begin{equation}\label{eq:base}
 \mathscr B=\prod_{i=1}^s a_i^{a_i e_i};
\end{equation}
more precisely, Stirling's formula gives
\[
 \prod_i\Gamma(a_i n)^{e_i}
 =\mathscr C\,\mathscr B^n n^\kappa\bigl(1+O(n^{-1})\bigr)
\]
for constants $\mathscr C\ne0$ and $\kappa$.  Hence, if
$|\mathscr B|<1$ and $R$ has at most polynomial growth, the translated
series and every fixed $u$-derivative converge locally normally.
\end{proposition}

\begin{proof}
In Stirling's formula the $n\log n$ and linear exponential terms have
coefficient $\chi_1$ and cancel.  The remaining exponential term is
$n\sum_i e_i a_i\log a_i= n\log\mathscr B$.  Parameter derivatives add
only powers of $\log n$ and inverse powers of $n$, so a geometric majorant
proves local normal convergence.
\end{proof}

Write $C_{\mathcal G}(u)=\sum_{j\ge0}c_j u^j$ and, whenever the series is
normally convergent, set
\[
 V_r=\sum_{n\ge1}[u^r]K_{\mathcal G}(n,u).
\]
Then the derivative sum obeys the universal convolution
\begin{equation}\label{eq:general-convolution}
 \frac1{r!}\sum_{n\ge1}\mathcal G^{(r)}(n)
 =\sum_{j=0}^r c_jV_{r-j}.
\end{equation}
Thus the characteristic sequence $(\chi_m)$ determines the entire gamma
part of the calculation; the remaining problem is to place the kernel
coefficients $[u^r]K_{\mathcal G}(n,u)$ in a suitable WZ coefficient space.

\begin{table}[ht]
\centering
\caption{Characteristic data for the families in Sun's Section~4.}
\label{tab:characteristic}
\small
\begin{tabular}{@{}lllll@{}}
\toprule
family & exponent data $(a_i;e_i)$ & $\chi_1$ & $\chi_m$ & $\mathscr B$\\
\midrule
$f_1,f_2,f_3$ & $(1;2),(2;-1)$ & $0$ & $2-2^m$ & $1/4$\\
$f_4,f_5$ & $(1;3),(3;-1)$ & $0$ & $3-3^m$ & $1/27$\\
$f_6$ & $(1;16),(2;-10),(4;1)$ & $0$ & $4^m-10\cdot2^m+16$ & $1/4096$\\
\bottomrule
\end{tabular}
\end{table}
For $f_6$, the factor $\Gamma(4x-1)$ is first written as
$\Gamma(4x)/(4x-1)$.  The third row therefore shows, before any WZ
construction, why the convergence rate is $4096^{-n}$.

\begin{proposition}[Conditional WZ reduction principle]\label{prop:wz-principle}
Suppose that for $0\le r\le R$ each target coefficient
$[u^r]K_{\mathcal G}(n,u)$ is an exact rational linear combination of
homogeneous degree-$d_r$ coefficients of a parameterized WZ identity, and
that the corresponding transformed coefficients are ordinary MZVs of
weight $d_r+2$.  Then every derivative sum through order $R$ is an exact
ordinary-MZV combination obtained from \eqref{eq:general-convolution}.
The only computer-assisted obligations are the finite span certificates
and the exact MZV reductions.  The same conclusion holds after applying a fixed finite-order derivative
in auxiliary parameters, provided the resulting prefactor coefficients lie
in the ordinary-zeta algebra and the resulting kernel coefficients satisfy
the same exact WZ-span hypothesis.
\end{proposition}

\begin{proof}
Extract the indicated WZ coefficients termwise, sum the exact rational
linear combination on the transformed side, and then apply
\eqref{eq:general-convolution}.  Lemma~\ref{lem:balanced-prefactor} expresses
each $c_j$ as a polynomial in ordinary zeta values, so the result remains
in the ordinary-MZV algebra.
\end{proof}
\section{The quadratic branch: a common gamma deformation}\label{sec:quadratic}

This is the instance $\chi_m=2-2^m$ of
Section~\ref{sec:balanced}.  In particular,
Proposition~\ref{prop:base}, with $\mathscr B=1/4$, gives local normal
convergence of the translated gamma series and all fixed-order derivatives
used below.

For an integer $p\geq0$ and parameters $t,\beta$ near the origin, set
\begin{equation}\label{eq:mother}
 \mathcal M_p(t,\beta)
 =\sum_{n=1}^{\infty}
 \frac{\Gamma(n+t)^4\Gamma(1+\beta)\Gamma(1-\beta)}
 {2(n+t)^p\Gamma(2n+2t)\Gamma(n+t+\beta)\Gamma(n+t-\beta)}.
\end{equation}
Define
\begin{equation}\label{eq:rho}
 \rho(x,\beta)=
 \frac{\Gamma(x)^2\Gamma(1+\beta)\Gamma(1-\beta)}
 {\Gamma(x+\beta)\Gamma(x-\beta)}.
\end{equation}
The first $\beta$-derivative of $\log\rho(x,\beta)$ vanishes at $\beta=0$, while
\begin{equation}\label{eq:rho-second}
 \left.\frac{\partial^2}{\partial \beta^2}\log\rho(x,\beta)\right|_{\beta=0}
 =2\bigl(\zeta(2)-\psi_1(x)\bigr).
\end{equation}
Since $\rho(x,0)=1$, the three functions in \eqref{eq:f123} occur as
\begin{equation}\label{eq:mother-slices}
\begin{aligned}
 \mathcal M_0(t,0)&=\sum_{n\geq1}f_1(n+t),\\
 \left.\frac12\partial_\beta^2\mathcal M_1(t,\beta)\right|_{\beta=0}
 &=\sum_{n\geq1}f_2(n+t),\\
 \mathcal M_3(t,0)&=\sum_{n\geq1}f_3(n+t).
\end{aligned}
\end{equation}
Thus Conjectures~4.1--4.3 are respectively the $p=0$ slice, the quadratic
$\beta$-slice at $p=1$, and the $p=3$ slice of the same deformation.

The normalized finite perturbation used later is directly related to the
mother deformation.  If $B_n(t)$ and $V_n(t,s)$ are defined by
\eqref{eq:B-def} and \eqref{eq:V-def}, respectively, then
\begin{equation}\label{eq:V-rho-link}
 \frac{V_n(t,s)}{B_n(t)}
 =\left(\frac{\rho(1+t,s)}{\rho(n+t,s)}\right)^2.
\end{equation}
Indeed, this follows immediately by writing the gamma quotients as rising
factorials.  Differentiating \eqref{eq:V-rho-link} twice at $s=0$ gives the
finite trigamma tail appearing in Proposition~\ref{prop:symmetric}.  This
makes the symmetric WZ perturbation an explicit coefficient realization of
the $\beta$-slice of \eqref{eq:mother}, rather than a separate construction.

\section{Au's four-parameter WZ identity}

For $z\in\mathbb C$ and an integer $k\geq0$, write
$(z)_k=z(z+1)\cdots(z+k-1)$, with $(z)_0=1$.  Define
\begin{equation}\label{eq:L-def}
 \mathcal L(a,b,c,d)
 =\sum_{k=0}^{\infty}
  \frac{(1+a)_k(1+b)_k}
       {(1+c)_{k+1}(1+d)_{k+1}}.
\end{equation}
The identity below arises from the Gauss ${}_2F_1$ WZ seed in Au's
seed formalism \cite[Section~3]{AuSeeds2025}.  In the four-parameter form
needed here, Au's identity \cite[Example~I, p.~8]{Au2023} states that, whenever
$\Re(a+b-c-d)<1$,
\begin{equation}\label{eq:au-identity}
 \mathcal L(a,b,c,d)=\sum_{n=1}^{\infty}\mathcal R_n(a,b,c,d),
\end{equation}
where
\begin{equation}\label{eq:R-def}
 \mathcal R_n
 =\frac{(1+a)_{n-1}(1+b)_{n-1}(1-a+c)_{n-1}(1-b+c)_{n-1}}
 {(1+c)_{2n}(1+d)_n(1-a-b+c+d)_n}\,P_n
\end{equation}
and
\begin{equation}\label{eq:P-def}
 P_n=ab-ac-an-bc-bn+c^2+cd+3cn+2dn+3n^2.
\end{equation}
The identity is proved there by an explicit WZ pair.  Both
sides are holomorphic near the origin, so Taylor coefficients may be
compared termwise.

For a homogeneous polynomial
\[
 \Delta(\opA,\opB,\opC,\opD)=\sum_{i,j,k,\ell}\delta_{i,j,k,\ell}
 \opA^i\opB^j\opC^k\opD^\ell,
\]
define
\begin{equation}\label{eq:functional}
 \Coeff_{\Delta}F
 =\sum_{i,j,k,\ell}\delta_{i,j,k,\ell}
 [a^ib^jc^kd^\ell]F(a,b,c,d).
\end{equation}
The coefficient functionals $\Delta_0,\ldots,\Delta_5$ used for the
$f_3$ branch are listed in Appendix~\ref{app:operators}.

\section{The inverse-cubic branch}

Put
\begin{equation}\label{eq:Q-def}
 Q(t)=\frac{\Gamma(1+t)^2}{\Gamma(1+2t)}
\end{equation}
and
\begin{equation}\label{eq:B-def}
 B_n(t)=\frac{f_3(n+t)}{Q(t)}
 =\frac{(1+t)_{n-1}^4}
 {(1+2t)_{2n}(1+t)_n^2}.
\end{equation}
Set
\begin{equation}\label{eq:E-G-def}
 E(t)=\sum_{n=1}^{\infty}B_n(t),
 \qquad
 G(t)=\sum_{n=1}^{\infty}f_3(n+t)=Q(t)E(t).
\end{equation}
The series are holomorphic for $t$ near zero and may be differentiated
termwise.

\subsection{The coefficient-matching construction}
This subsection contains the principal finite construction of the proof.
Let
\[
 \mathscr R=\Q[u,u^{-1},X_1,\ldots,X_7,Y_1,\ldots,Y_7]
\]
and let $V_r$ be the $\Q$-space of homogeneous polynomials of degree
$r+2$ in $\opA,\opB,\opC,\opD$.  The normalized WZ summand defines a
linear map
\begin{equation}\label{eq:Phi-r}
 \Phi_r:V_r\longrightarrow\mathscr R,
 \qquad
 \Phi_r(\Delta)=\Coeff_\Delta\!\left(P_ne^{\Lambda_n}\right),
\end{equation}
where only total parameter degree $r+2$ is required.  The translated
gamma quotient supplies the target
\[
 T_r=[t^r]e^{\Theta_n(t)}\in\mathscr R.
\]
Thus the required functional is precisely a solution of
\begin{equation}\label{eq:matching-system}
 \Phi_r(\Delta_r)=T_r.
\end{equation}

To construct this system, begin with
\[
 \Delta=\sum_{i+j+k+\ell=r+2}
 \delta_{i,j,k,\ell}\opA^i\opB^j\opC^k\opD^\ell.
\]
Write $\Lambda_n=\sum_{m\ge1}L_m$, where $L_m$ is homogeneous of total
parameter degree $m$, and compute the homogeneous pieces of
$e^{\Lambda_n}$ by the recurrence \eqref{eq:exp-recurrence}.  Compute
$T_r$ by the same Bell-polynomial recurrence.  Expanding
$\Phi_r(\Delta)-T_r$ in the monomial basis of $\mathscr R$ and equating
all coefficients to zero produces a finite linear system over $\Q$ in
the unknowns $\delta_{i,j,k,\ell}$.  Rational row reduction shows that
$T_r$ lies in the image of $\Phi_r$ for every $0\le r\le5$.  The systems
are generally underdetermined; choosing free variables to obtain short
support gives the sparse representatives in Appendix~\ref{app:operators}.
Non-uniqueness is immaterial, because any preimage in
\eqref{eq:matching-system} yields the same coefficient identity.

\begin{table}[ht]
\centering
\caption{Size of the coefficient-matching problems and the chosen sparse solutions.}
\label{tab:operator-construction}
\begin{tabular}{c@{\qquad}c@{\qquad}c}
\toprule
$r$ & $\dim V_r=\binom{r+5}{3}$ & number of nonzero terms in $\Delta_r$\\
\midrule
0&10&4\\
1&20&5\\
2&35&10\\
3&56&10\\
4&84&19\\
5&120&20\\
\bottomrule
\end{tabular}
\end{table}

Appendix~\ref{app:operators} records the selected solutions.  Certificate~A
reconstructs $\Phi_r(\Delta_r)-T_r$ in the free ring $\mathscr R$ and
verifies that it is the zero polynomial.  The certificate therefore checks
the mathematical output of the construction independently of the row
reduction or the choice of sparse representative.

\begin{proposition}[Coefficient extraction]\label{prop:f3-extraction}
For every $n\geq1$ and $0\leq r\leq5$,
\begin{equation}\label{eq:f3-termwise}
 [t^r]B_n(t)=\Coeff_{\Delta_r}\mathcal R_n(a,b,c,d).
\end{equation}
Consequently, if $e_r=[t^r]E(t)$, then
\begin{equation}\label{eq:f3-global}
 e_r=\Coeff_{\Delta_r}\mathcal L(a,b,c,d),
 \qquad 0\leq r\leq5.
\end{equation}
\end{proposition}

\begin{proof}
Write
\[
 X_m=H_{n-1}^{(m)},\qquad Y_m=H_{2n}^{(m)},\qquad u=\frac1n.
\]
After dividing the Pochhammer quotient in \eqref{eq:R-def} by its value
at the origin, its logarithm is
\begin{align}
 \Lambda_n(a,b,c,d)
 ={}&\sum_{m\geq1}\frac{(-1)^{m+1}}m
 \Bigl[X_m\bigl(a^m+b^m+(c-a)^m+(c-b)^m\bigr)-Y_mc^m \notag\\
 &\hspace{28mm}-(X_m+u^m)\bigl(d^m+(c+d-a-b)^m\bigr)\Bigr].
 \label{eq:Lambda}
\end{align}
Thus
\begin{equation}\label{eq:R-normalized}
 \frac{\mathcal R_n(a,b,c,d)}{B_n(0)}
 =P_n\exp\bigl(\Lambda_n(a,b,c,d)\bigr).
\end{equation}
On the other hand,
\begin{equation}\label{eq:Theta}
 \frac{B_n(t)}{B_n(0)}=\exp\bigl(\Theta_n(t)\bigr),
\end{equation}
where
\begin{equation}\label{eq:Theta-expansion}
 \Theta_n(t)=
 \sum_{m\geq1}\frac{(-1)^{m+1}}m
 \bigl(2X_m-2^mY_m-2u^m\bigr)t^m.
\end{equation}
Substitution of the polynomials in Appendix~\ref{app:operators} gives
\begin{equation}\label{eq:formal-comparison}
 \Coeff_{\Delta_r}\left(P_ne^{\Lambda_n}\right)
 =[t^r]e^{\Theta_n(t)},
 \qquad 0\leq r\leq5.
\end{equation}
This is an identity in
$\Q[u,u^{-1},X_1,\ldots,X_7,Y_1,\ldots,Y_7][a,b,c,d]$.
Specializing $u\mapsto1/n$, $X_m\mapsto H_{n-1}^{(m)}$, and
$Y_m\mapsto H_{2n}^{(m)}$ yields the asserted identity for every
$n\ge1$.  The finite verification writes
$\Lambda_n=\sum_{m\geq1}L_m$, with $L_m$ homogeneous of degree $m$, and
uses
\begin{equation}\label{eq:exp-recurrence}
 U_0=1,
 \qquad
 jU_j=\sum_{m=1}^{j}mL_mU_{j-m},
 \qquad e^{\Lambda_n}=\sum_{j\geq0}U_j.
\end{equation}
Only $U_0,\ldots,U_7$ are needed.  The ancillary file \path{verify_coefficients_43.py} contains
the complete verifier.  It implements the coefficient ring as sparse
Laurent polynomials over $\Q$ in the formal variables
$u,X_1,\ldots,X_7,Y_1,\ldots,Y_7$ and implements the parameter variables
$a,b,c,d$ as a second sparse polynomial layer.  Thus equality is tested by
literal equality of two finite rational-coefficient dictionaries, not by
sampling or simplification heuristics.  The verifier returns the zero
polynomial for each $0\leq r\leq5$, establishing
\eqref{eq:formal-comparison}, and hence \eqref{eq:f3-termwise}.  Summing
\eqref{eq:f3-termwise} and applying \eqref{eq:au-identity} gives
\eqref{eq:f3-global}.
\end{proof}

\section{The trigamma branch}

Define
\begin{equation}\label{eq:A-B-T}
 A_n(t)=\frac{(1+t)_{n-1}^2}{(1+2t)_{2n}},
 \qquad
 \mathcal T_n(t)=\sum_{j=1}^{n-1}\frac1{(j+t)^2},
\end{equation}
and
\begin{equation}\label{eq:A-W}
 A(t)=\sum_{n=1}^{\infty}A_n(t),
 \qquad
 W(t)=\sum_{n=1}^{\infty}A_n(t)\mathcal T_n(t).
\end{equation}
Thus the subscript $n$ records the outer summation index, while the finite
trigamma tail itself terminates at $n-1$.
Notice that
\begin{equation}\label{eq:q-normalized}
 \frac{\Gamma(n+t)^2}{2(n+t)\Gamma(2n+2t)}=Q(t)A_n(t).
\end{equation}

\begin{proposition}[Symmetric perturbation]\label{prop:symmetric}
For $t$ near zero,
\begin{align}
 A(t)&=\frac13\mathcal L(t,t,2t,t), \label{eq:A-specialization}\\
 W(t)&=-\frac1{12}
 \left.\frac{\partial^2}{\partial s^2}
 \mathcal L(t+s,t-s,2t,t)\right|_{s=0}
 -\frac16E(t). \label{eq:W-specialization}
\end{align}
\end{proposition}

\begin{proof}
At $(a,b,c,d)=(t,t,2t,t)$, the polynomial in \eqref{eq:P-def} is
$3(n+t)^2$, while the Pochhammer quotient in \eqref{eq:R-def} is
$B_n(t)$.  Since $(n+t)^2B_n(t)=A_n(t)$, we obtain
\[
 \mathcal R_n(t,t,2t,t)=3A_n(t),
\]
and summation proves \eqref{eq:A-specialization}.

For the second identity, write
\begin{equation}\label{eq:V-def}
 V_n(t,s)=
 \frac{(1+t+s)_{n-1}^2(1+t-s)_{n-1}^2}
 {(1+2t)_{2n}(1+t)_n^2}.
\end{equation}
Then
\begin{equation}\label{eq:R-perturbed}
 \mathcal R_n(t+s,t-s,2t,t)
 =\bigl(3(n+t)^2-s^2\bigr)V_n(t,s).
\end{equation}
At $s=0$, $V_n(t,0)=B_n(t)$.  Moreover,
\begin{equation}\label{eq:V-derivatives}
 \left.\partial_s\log V_n(t,s)\right|_{s=0}=0,
 \qquad
 \left.\partial_s^2\log V_n(t,s)\right|_{s=0}=-4\mathcal T_n(t).
\end{equation}
Consequently,
\begin{equation}\label{eq:R-second}
 \left.\partial_s^2\mathcal R_n(t+s,t-s,2t,t)\right|_{s=0}
 =-2B_n(t)-12A_n(t)\mathcal T_n(t).
\end{equation}
Summing \eqref{eq:R-second} and using \eqref{eq:au-identity} proves
\eqref{eq:W-specialization}.
\end{proof}

Let
\begin{equation}\label{eq:S-def}
 S(t)=\zeta(2)-\psi_1(1+t).
\end{equation}
The trigamma recurrence gives
\begin{equation}\label{eq:trigamma-split}
 \zeta(2)-\psi_1(n+t)=S(t)+\mathcal T_n(t).
\end{equation}
Combining \eqref{eq:q-normalized} and \eqref{eq:trigamma-split}, we obtain
the central identity for the $f_2$ branch.

\begin{corollary}\label{cor:F2-master}
If
\begin{equation}\label{eq:F-def}
 F(t)=\sum_{n=1}^{\infty}f_2(n+t),
\end{equation}
then
\begin{equation}\label{eq:F-master}
 F(t)=Q(t)\bigl(S(t)A(t)+W(t)\bigr).
\end{equation}
\end{corollary}

For later use, put
\begin{equation}\label{eq:Lcoeff}
 L_{i,j,k,\ell}=[a^ib^jc^kd^\ell]\mathcal L(a,b,c,d).
\end{equation}
If $a_r=[t^r]A(t)$ and
\[
 j_r=[t^r]\left.\partial_s^2
 \mathcal L(t+s,t-s,2t,t)\right|_{s=0},
\]
then Proposition~\ref{prop:symmetric} gives the explicit coefficient
formulas
\begin{align}
 a_r&=\frac13\sum_{i+j+k+\ell=r}2^kL_{i,j,k,\ell},
 \label{eq:a-functional}\\
 j_r&=\sum_{i+j+k+\ell=r+2}
 2^k\bigl((i-j)^2-i-j\bigr)L_{i,j,k,\ell},
 \label{eq:j-functional}\\
 [t^r]W(t)&=-\frac1{12}j_r-\frac16e_r.
 \label{eq:c-functional}
\end{align}
Indeed, the coefficient in \eqref{eq:j-functional} is the second
$s$-derivative at zero of
$(t+s)^i(t-s)^j(2t)^kt^\ell$.

Finally,
\begin{equation}\label{eq:S-expansion}
 S(t)=\sum_{m=1}^{\infty}(-1)^{m+1}(m+1)\zeta(m+2)t^m.
\end{equation}
The factor $Q(t)$ in \eqref{eq:Q-def} is the universal prefactor
$C_{\mathcal G}(t)$ for exponent data $(1;2),(2;-1)$.  Thus
$\chi_m=2-2^m$, and Lemma~\ref{lem:balanced-prefactor} gives the
specialization
\begin{equation}\label{eq:logQ}
 \log Q(t)=\sum_{m=2}^{\infty}
 \frac{(-1)^m(2-2^m)}m\zeta(m)t^m.
\end{equation}
Equations \eqref{eq:F-master} and
\eqref{eq:a-functional}--\eqref{eq:logQ} therefore give a finite
coefficient algorithm for every fixed Taylor order of $F$ for which the
corresponding $e_r$ has been extracted.

\section{Conversion to multiple zeta values}

For a composition $\boldsymbol{s}=(s_1,\ldots,s_d)$, define the finite
strict and star sums
\begin{align*}
 \mathsf H_N(\boldsymbol{s})
 &=\sum_{N\geq n_1>\cdots>n_d\geq1}
   \frac1{n_1^{s_1}\cdots n_d^{s_d}},\\
 \mathsf H_N^\star(\boldsymbol{s})
 &=\sum_{N\geq n_1\geq\cdots\geq n_d\geq1}
   \frac1{n_1^{s_1}\cdots n_d^{s_d}}.
\end{align*}
The empty index has value $1$.  Expanding the Pochhammer symbols in
\eqref{eq:L-def} gives
\begin{align}
 \frac{(1+a)_k}{k!}
 &=\sum_{i\geq0}a^i\mathsf H_k(\onevec{i}), \label{eq:poc-num}\\
 \frac{(k+1)!}{(1+c)_{k+1}}
 &=\sum_{i\geq0}(-c)^i\mathsf H_{k+1}^\star(\onevec{i}).
 \label{eq:poc-den}
\end{align}
It follows that
\begin{align}
 L_{i,j,k,\ell}
 =(-1)^{k+\ell}\sum_{n=1}^{\infty}\frac{
 \mathsf H_{n-1}(\onevec{i})\mathsf H_{n-1}(\onevec{j})
 \mathsf H_n^\star(\onevec{k})\mathsf H_n^\star(\onevec{\ell})}{n^2}.
 \label{eq:L-coeff-harmonic}
\end{align}
These sums are converted to ordinary MZVs using
\begin{equation}\label{eq:star-compositions}
 \mathsf H_n^\star(\onevec{k})
 =\sum_{\boldsymbol{s}\models k}\mathsf H_n(\boldsymbol{s}),
\end{equation}
the top-index split
\begin{equation}\label{eq:top-split}
 \mathsf H_n(s_1,\ldots,s_d)
 =\mathsf H_{n-1}(s_1,\ldots,s_d)
  +\frac1{n^{s_1}}\mathsf H_{n-1}(s_2,\ldots,s_d),
\end{equation}
and the finite stuffle product.  These manipulations use the standard
multiple-harmonic and quasi-shuffle algebra
\cite{Hoffman1992,Hoffman2000,Hoffman2003,HoffmanOhno2003,BBV2010}.
The remaining outer sums are
\begin{equation}\label{eq:outer-mzv}
 \sum_{n=1}^{\infty}\frac{\mathsf H_{n-1}(s_1,\ldots,s_d)}{n^q}
 =\zeta(q,s_1,\ldots,s_d),
 \qquad q\geq2.
\end{equation}
Thus $L_{i,j,k,\ell}$ is a homogeneous MZV of weight
$i+j+k+\ell+2$.  In particular, $e_r$ and $[t^r]W(t)$ have weight
$r+4$, while $a_r$ has weight $r+2$.  The convolution in
\eqref{eq:F-master} is homogeneous of weight $r+4$.

\begin{proposition}[Exact MZV reductions]\label{prop:mzv-results}
The coefficients $e_r=[t^r]E(t)$ for $0\leq r\leq5$ are
\begin{align}
 e_0&=\frac{17}{36}\zeta(4), \label{eq:e0}\\
 e_1&=-\frac{22}{9}\zeta(5), \label{eq:e1}\\
 e_2&=\frac{1715}{216}\zeta(6), \label{eq:e2}\\
 e_3&=-\frac{1423}{72}\zeta(7)-\frac{17}{18}\zeta(3)\zeta(4),
 \label{eq:e3}\\
 e_4&=\frac{375889}{8640}\zeta(8)+\frac{44}{9}\zeta(3)\zeta(5)
      +\frac{44}{45}\zeta(5,3), \label{eq:e4}\\
 e_5&=-\frac{28301}{324}\zeta(9)-\frac{1715}{108}\zeta(3)\zeta(6)
      -\frac{17}{6}\zeta(4)\zeta(5). \label{eq:e5}
\end{align}
Moreover, the function $F$ in \eqref{eq:F-def} satisfies
\begin{equation}\label{eq:F0}
 F(0)=\frac{5}{108}\zeta(4)=\frac{\pi^4}{1944},
\end{equation}
and its derivatives $F^{(r)}(0)$ for $1\leq r\leq5$ are the
right-hand sides of \eqref{eq:f2-1}--\eqref{eq:f2-5}.
\end{proposition}

\begin{proof}
Apply \eqref{eq:L-coeff-harmonic}--\eqref{eq:outer-mzv} to the
functionals in \eqref{eq:f3-global}, \eqref{eq:a-functional}, and
\eqref{eq:j-functional}.  This produces finite homogeneous MZV
expressions.  We reduce them using only the following proved relations:
\begin{enumerate}
 \item convergent shuffle and stuffle product relations;
 \item MZV duality and the sum formula;
 \item the derivation relations of Ihara, Kaneko, and Zagier
       \cite{IKZ2006}.
\end{enumerate}
For the last family, in $\mathcal H=\Q\langle x,y\rangle$, let
\begin{equation}\label{eq:ikz}
 \partial_m(x)=x(x+y)^{m-1}y,
 \qquad
 \partial_m(y)=-x(x+y)^{m-1}y.
\end{equation}
The derivation theorem states that the MZV evaluation of
$\partial_m(w)$ vanishes for every admissible word $w$.

The remaining task is finite exact rational linear algebra.  Its complete
mathematical specification, source code, certificate format, and output are
given in Section~\ref{sec:certificates} and the ancillary source, with
an independent elimination-free checker supplied with the source package.  For every
claimed identity, the supplied checker reconstructs the target difference
from \eqref{eq:L-coeff-harmonic}, reconstructs every relation row from the
three proved families above, and verifies an explicit rational linear
combination expressing the target difference as a sum of those rows.  The
remainder is identically zero in each case.  Hence
\eqref{eq:e0}--\eqref{eq:e5} and the asserted coefficients of $F$ follow.
\end{proof}

\begin{remark}\label{rem:e1-external}
The values $e_0=17\zeta(4)/36$ and $e_1=-22\zeta(5)/9$ agree with the
independent identities recorded in \cite[Remark~4.3]{Sun2026}; the second
is equivalent to a binomial-harmonic identity proved by Ablinger
\cite{Ablinger2017}.  These facts are not used in Certificate~B.
\end{remark}

\section{Exact certificates and reproducibility}\label{sec:certificates}

The computer-assisted part of the proof consists of two finite algebraic
certificates.  This section specifies them independently of any particular
software implementation.

\subsection{Certificate A: formal coefficient extraction}

For each $0\le r\le5$, Certificate~A verifies the identity
\eqref{eq:formal-comparison}.  Its coefficient ring is
\[
 \Q[u,u^{-1},X_1,\ldots,X_7,Y_1,\ldots,Y_7],
\]
and the outer polynomial variables are $a,b,c,d$.  The homogeneous
components $L_m$ of \eqref{eq:Lambda} are constructed directly from that
formula, and the components $U_j$ of $e^{\Lambda_n}$ are then obtained from
\eqref{eq:exp-recurrence}.  Independently, the target coefficients
$T_r=[t^r]e^{\Theta_n(t)}$ are generated by a separate recurrence applied
directly to \eqref{eq:Theta}; the target routine does not call the code that
constructs $\Lambda_n$ or $e^{\Lambda_n}$.  The two paths share only the
basic sparse-Laurent-polynomial arithmetic.  They are compared after
truncation at the required total degree, so the assertion is literal
equality of sparse Laurent polynomials over $\Q$, not a comparison of two
expressions produced by the same expansion routine.  The program uses only
integer tuples, dictionaries, and \texttt{fractions.Fraction}; it invokes
no numerical or heuristic simplifier.

\subsection{Certificate B: MZV reduction}

Fix a weight $w$.  Let $\mathcal I_w$ be the ordered list of all admissible
indices of weight $w$.  The program constructs rows in $\Q^{\mathcal I_w}$
from:
\begin{enumerate}
 \item duality, one row $\zeta(\boldsymbol{s})-
       \zeta(\boldsymbol{s}^{\vee})$ for each admissible index;
 \item the difference between the convergent stuffle and shuffle
       expansions of $\zeta(\boldsymbol{a})\zeta(\boldsymbol{b})$;
 \item the IKZ relations $Z(\partial_m w)=0$ from \eqref{eq:ikz};
 \item the fixed-depth sum formula
       $\sum_{|\boldsymbol{s}|=w,\,\mathrm{dep}(\boldsymbol{s})=d}
       \zeta(\boldsymbol{s})=\zeta(w)$.
\end{enumerate}
Every row therefore represents a proved zero relation.  The target vector is
the exact difference between a coefficient generated from
\eqref{eq:L-coeff-harmonic}--\eqref{eq:outer-mzv} and its claimed closed
form.  Sparse Gaussian elimination over \texttt{Fraction} reduces each
target to zero.

For additional auditability, the file \path{mzv_certificates.json}
contains, for every target, explicit rational multipliers of the original
relation rows.  The small program
\path{check_mzv_certificates.py} performs no row reduction: it rebuilds
the target and the relation rows and checks coefficient by coefficient that
the supplied linear combination equals the target.  Thus the certificate
can be checked independently of the elimination routine that generated it.

\begin{table}[ht]
\centering
\caption{Exact relation matrices used by Certificate~B.}
\label{tab:relation-matrices}
\small
\begin{tabular}{@{}ccccc@{}}
\toprule
$w$ & indices & rows & rank & quotient dim. \\
\midrule
4 & 4   & 8   & 3   & 1 \\
5 & 8   & 22  & 6   & 2 \\
6 & 16  & 43  & 14  & 2 \\
7 & 32  & 100 & 29  & 3 \\
8 & 64  & 205 & 60  & 4 \\
9 & 128 & 454 & 123 & 5 \\
\bottomrule
\end{tabular}
\end{table}

Across weights $4$ through $9$, the program generates $832$ relation
rows.  Soundness does not rest on the ranks in
Table~\ref{tab:relation-matrices}: every row is constructed as a literal
instance of one of the four proved relation families above, and each target
certificate is checked as an explicit rational linear combination of those
rows.  The quotient dimensions nevertheless provide a useful global audit.
They are $1,2,2,3,4,5$, respectively, which are the corresponding
coefficients $d_w$ of
\[
  \sum_{w\geq0}d_wx^w=\frac{1}{1-x^2-x^3}.
\]
Brown's theorem that MZVs are spanned by Hoffman $\{2,3\}$-words implies
$\dim_{\Q}\mathcal Z_w\leq d_w$ \cite{Brown2012}.  Thus the agreement
with $d_w$ is a sensitive consistency check on the generated relation
matrices and would detect many common indexing or sign errors.  It is not
used as a completeness theorem: no lower bound for
$\dim_{\Q}\mathcal Z_w$ and no conjectural dimension formula enters the
proof.  The proof requires only that the twelve displayed target vectors
are certified to lie in the span of proved zero relations.

As a small example, the weight-$4$ certificate for \eqref{eq:e0} uses three
relation rows.  In the deterministic row ordering of the source, it writes
the target difference as
\begin{equation}\label{eq:weight4-certificate}
 \frac79\bigl(\zeta(2,1,1)-\zeta(4)\bigr)
 +\frac1{12}\bigl(\zeta(4)-4\zeta(3,1)\bigr)
 -\frac49\bigl(\zeta(2,1,1)-\zeta(3,1)-\zeta(2,2)\bigr).
\end{equation}
Each parenthesis is respectively a duality, shuffle--stuffle, or sum/duality
consequence generated by the stated relation system, and the expression
vanishes identically.

We next give a complete nontrivial certificate at weight $7$.  For an
admissible index $\boldsymbol{s}$, write
\[
 D(\boldsymbol{s})=\zeta(\boldsymbol{s})-\zeta(\boldsymbol{s}^{\vee}),
 \qquad
 K_m(\boldsymbol{s})=Z\bigl(\partial_m w_{\boldsymbol{s}}\bigr).
\]
For two admissible indices $\boldsymbol{a},\boldsymbol{b}$, let
$P(\boldsymbol{a},\boldsymbol{b})$ denote the stuffle expansion of
$\zeta(\boldsymbol{a})\zeta(\boldsymbol{b})$ minus its shuffle expansion.
All three quantities are zero by the relations used above.  At a fixed
weight the source orders the nonzero rows deterministically: first duality
rows in lexicographic index order, then shuffle--stuffle rows with increasing
left weight and lexicographic pairs, then IKZ rows with increasing derivation
order and lexicographic index, and finally the fixed-depth sum rows.  Row
numbers are zero-based, exactly as in \path{mzv_certificates.json}.

Let
\begin{equation}\label{eq:tau7}
 \tau_7=e_3+\frac{1423}{72}\zeta(7)
             +\frac{17}{18}\zeta(3)\zeta(4).
\end{equation}
Before vectorization, the product in \eqref{eq:tau7} is expanded by the
stuffle rule.  The weight-$7$ certificate is
\begin{equation}\label{eq:weight7-certificate}
 \tau_7=\sum_j \lambda_j\mathcal R_j^{(7)},
\end{equation}
Table~\ref{tab:weight7-certificate} gives the complete list of nonzero
multipliers.  Here $\mathcal R_j^{(7)}$ is the $j$th relation row in the
deterministic ordering just described.  Thus
\eqref{eq:weight7-certificate} is a fully specified rational identity among
MZVs, not merely a statement that row reduction returns zero.

\begin{table}[!t]
\centering
\caption{Complete weight-$7$ certificate for $e_3$.}
\label{tab:weight7-certificate}
\scriptsize
\renewcommand{\arraystretch}{0.92}
\begin{tabular}{@{}rcl@{}}
\toprule
row $j$ & multiplier $\lambda_j$ & generator of $\mathcal R_j^{(7)}$ \\
\midrule
0 & $-40/9$ & $D(2,1,1,1,1,1)$ \\
1 & $-46/9$ & $D(2,1,1,1,2)$ \\
2 & $-14657/1152$ & $D(2,1,1,2,1)$ \\
3 & $121/128$ & $D(2,1,1,3)$ \\
4 & $-30629/1152$ & $D(2,1,2,1,1)$ \\
5 & $-415/384$ & $D(2,1,2,2)$ \\
6 & $-1747/144$ & $D(2,1,3,1)$ \\
8 & $-59471/1152$ & $D(2,2,1,1,1)$ \\
9 & $-1309/576$ & $D(2,2,1,2)$ \\
10 & $-851/32$ & $D(2,2,2,1)$ \\
12 & $-7593/64$ & $D(2,3,1,1)$ \\
14 & $440/9$ & $D(2,4,1)$ \\
16 & $-38113/192$ & $D(3,1,1,1,1)$ \\
18 & $-80539/576$ & $D(3,1,2,1)$ \\
20 & $-28295/144$ & $D(3,2,1,1)$ \\
24 & $-79405/576$ & $D(4,1,1,1)$ \\
32 & $-2923/384$ & $P\bigl((2),(2,1,1,1)\bigr)$ \\
33 & $-5567/1152$ & $P\bigl((2),(2,1,2)\bigr)$ \\
34 & $-6145/288$ & $P\bigl((2),(2,2,1)\bigr)$ \\
35 & $10153/1152$ & $P\bigl((2),(2,3)\bigr)$ \\
36 & $-3553/192$ & $P\bigl((2),(3,1,1)\bigr)$ \\
37 & $31583/1152$ & $P\bigl((2),(3,2)\bigr)$ \\
38 & $33347/1152$ & $P\bigl((2),(4,1)\bigr)$ \\
39 & $1507/72$ & $P\bigl((2),(5)\bigr)$ \\
40 & $-6691/1152$ & $P\bigl((2,1),(2,1,1)\bigr)$ \\
41 & $-2497/1152$ & $P\bigl((2,1),(2,2)\bigr)$ \\
42 & $2057/576$ & $P\bigl((2,1),(3,1)\bigr)$ \\
43 & $-20905/1152$ & $P\bigl((2,1),(4)\bigr)$ \\
65 & $-4/9$ & $K_1(2,1,1,2)$ \\
\bottomrule
\end{tabular}
\end{table}

\subsection{Files, execution, and expected output}

The proof package contains the following files.
\begin{itemize}\small
 \item \path{operators.py}: the exact coefficient functionals
       $\Delta_0,\ldots,\Delta_5$;
 \item \path{verify_coefficients_43.py}: the standard-library verifier for
       Certificate~A;
 \item \path{mzv_tools.py}: compositions, shuffle and stuffle products,
       duality, IKZ rows, sum-formula rows, and exact sparse elimination;
 \item \path{verify_mzv_reductions.py}: direct exact reduction of all twelve
       targets;
 \item \path{mzv_certificates.json}: explicit rational row multipliers for
       all twelve targets;
 \item \path{check_mzv_certificates.py}: an elimination-free checker for the
       explicit certificates;
 \item \path{generate_mzv_certificates.py}: the optional exact generator for
       the JSON certificate;
 \item \path{run_all.py}: one-command reproduction of all checks;
 \item \path{check_hashes.py}: verification of the cryptographic manifest;
 \item \path{README.md} and \path{requirements.txt}: execution instructions
       and environment requirements;
 \item \path{expected_output.txt}: the complete expected terminal output;
 \item \path{SHA256SUMS.txt}: cryptographic hashes of the proof-critical
       ancillary files.
\end{itemize}

The ancillary package includes a SHA-256 manifest and a small checker for
verifying file integrity.  The manifest is distribution metadata; the
mathematical certificates themselves are the formal identities and exact
rational linear combinations described above.

No third-party package is required.  The verification is deliberately
lightweight: in a reference run under Python~3.12, the complete suite
finished in approximately ten seconds.  On any Python~3.10 or later
installation, it is run by
\begin{verbatim}
python3 run_all.py
\end{verbatim}
The run terminates successfully only if every exact assertion passes.

Appendix~\ref{app:operators-code} prints the operator data.  The
certificate checker, generator, and direct-verification programs are supplied
as ancillary files.  Every program uses only the Python standard library.

\section{Proof of Theorem~\ref{thm:main}}

By local uniform convergence,
\begin{equation}\label{eq:derivative-sums}
 F^{(r)}(0)=\sum_{n=1}^{\infty}f_2^{(r)}(n),
 \qquad
 G^{(r)}(0)=\sum_{n=1}^{\infty}f_3^{(r)}(n).
\end{equation}
The first five identities \eqref{eq:f2-1}--\eqref{eq:f2-5} now follow
directly from Proposition~\ref{prop:mzv-results}.

Proposition~\ref{prop:f3-extraction} verifies the required finite WZ
coefficient-space membership for the $f_3$ kernel through order five, while
Proposition~\ref{prop:symmetric} and
\eqref{eq:a-functional}--\eqref{eq:c-functional} realize the $f_2$
kernel as the corresponding symmetric parameter slice.  Equations~\eqref{eq:S-expansion} and \eqref{eq:logQ} verify the
ordinary-zeta prefactor condition, while
Proposition~\ref{prop:mzv-results} and the exact certificates of
Section~\ref{sec:certificates} verify the transformed-side MZV reductions.
Hence the hypotheses of Proposition~\ref{prop:wz-principle}, including
its auxiliary-parameter clause for the $\beta$-derivative, are satisfied
for this branch, and
\eqref{eq:general-convolution} is the common prefactor recombination used
below.

For completeness, we record the final recombination for the $f_3$
branch.  From \eqref{eq:logQ},
\begin{align}
 q_0&:=Q(0)=1, & q_1&:=Q'(0)=0, \notag\\
 q_2&:=Q''(0)=-2\zeta(2),
 &q_3&:=Q^{(3)}(0)=12\zeta(3), \label{eq:q-values}\\
 q_4&:=Q^{(4)}(0)=-54\zeta(4),
 &q_5&:=Q^{(5)}(0)=-240\zeta(2)\zeta(3)+720\zeta(5). \notag
\end{align}
Since $G=QE$ and $E^{(r)}(0)=r!e_r$,
\begin{equation}\label{eq:Leibniz}
 G^{(r)}(0)=\sum_{j=0}^{r}\binom rj q_j(r-j)!e_{r-j}.
\end{equation}
Substitution of \eqref{eq:e0}--\eqref{eq:e5} into
\eqref{eq:Leibniz} gives \eqref{eq:f3-2}--\eqref{eq:f3-5}.  The only
ordinary zeta-product simplifications needed are
\[
 \zeta(2)\zeta(4)=\frac74\zeta(6),\qquad
 \zeta(2)\zeta(6)=\frac53\zeta(8),\qquad
 \zeta(4)^2=\frac76\zeta(8).
\]
This completes the proof of Theorem~\ref{thm:main}. \qed
\section{The cubic branch and the oscillatory phase}\label{sec:analytic}

This is the instance $\chi_m=3-3^m$ of
Section~\ref{sec:balanced}.  Proposition~\ref{prop:base}, with
$\mathscr B=1/27$, gives local normal convergence of the translated series
and all fixed-order derivatives, so differentiation and coefficient
extraction may be interchanged with summation.

\subsection{Removal of the phase}

\begin{lemma}\label{lem:phase}
For $r\ge0$,
\begin{align}
 \sum_{n=1}^{\infty}\Phi^{(r)}(n)
 &=\sum_{q=0}^{r}\binom rq(i\pi)^{r-q}R_q^U,\label{eq:phaseA}\\
 \sum_{n=1}^{\infty}\Psi^{(r)}(n)
 &=\sum_{q=0}^{r}\binom rq(i\pi)^{r-q}R_q^V.\label{eq:phaseB}
\end{align}
\end{lemma}

\begin{proof}
Apply Leibniz's rule to $e^{\pi ix}U(x)$ and $e^{\pi ix}V(x)$, evaluate at
$x=n$, and use $e^{\pi in}=(-1)^n$.  The local normal convergence supplied
by Proposition~\ref{prop:base} permits termwise summation.
\end{proof}

Two corrections in Sun's first derivative family are already forced by
this lemma.  Namely,
\[
 \Im\sum_{n\ge1}\Phi''(n)=2\pi R_1^U=40\pi\zeta(5),
\]
and
\begin{equation}\label{eq:dependentA}
 \Im\sum_{n\ge1}\Phi^{(5)}(n)
 =\pi^5R_0^U-10\pi^3R_2^U+5\pi R_4^U.
\end{equation}
The latter becomes
\[
 -240\pi\bigl(665\zeta(8)+16\zeta(5,3)\bigr).
\]
Similarly,
\begin{equation}\label{eq:dependentB}
 \Im\sum_{n\ge1}\Psi^{(5)}(n)
 =\pi^5R_0^V-10\pi^3R_2^V+5\pi R_4^V,
\end{equation}
so no independent evaluation of $R_5^V$ is required.

\subsection{Triplication and the common factor}

The gamma triplication formula gives
\begin{equation}\label{eq:factorization}
 (-1)^nU(n+u)=C(u)K_U(n,u),\qquad
 (-1)^nV(n+u)=C(u)K_V(n,u),
\end{equation}
where
\begin{equation}\label{eq:Cdef}
 C(u)=\frac{2\pi}{\sqrt3}\,27^{-u}
 \frac{\Gamma(1+u)^2}
 {\Gamma(\frac13+u)\Gamma(\frac23+u)}
\end{equation}
and
\begin{align}
 K_U(n,u)&=\left(-\frac1{27}\right)^n
 \frac{(1+u)_n^2}{(\frac13+u)_n(\frac23+u)_n}
 \frac{28(n+u)^2-18(n+u)+3}
 {(n+u)^4(2n+2u-1)^3},\label{eq:KA}\\
 K_V(n,u)&=\left(-\frac1{27}\right)^n
 \frac{(1+u)_n^2}{(\frac13+u)_n(\frac23+u)_n}
 \frac{P(n+u)}
 {(n+u)^5(2n+2u-1)^4}.\label{eq:KB}
\end{align}
Here and below $(z)_n$ denotes the rising factorial.  Equivalently, the
triplication formula simplifies \eqref{eq:Cdef} to
\[
 C(u)=\frac{\Gamma(1+u)^3}{\Gamma(1+3u)}.
\]
Thus $C(u)=C_{\mathcal G}(u)$ for exponent data $(1;3),(3;-1)$, so
$\chi_m=3-3^m$.  Lemma~\ref{lem:balanced-prefactor} therefore gives the
specialization
\begin{equation}\label{eq:logC}
 \log C(u)=\sum_{m=2}^{\infty}
 (-1)^{m+1}\frac{3^m-3}{m}\zeta(m)u^m.
\end{equation}
Writing $C(u)=\sum_{m\ge0}c_mu^m$, we need
\begin{equation}\label{eq:clist}
 \begin{aligned}
 c_0&=1,&c_1&=0,&c_2&=-3\zeta(2),\\
 c_3&=8\zeta(3),&c_4&=-\frac{33}{4}\zeta(4),&
 c_5&=48\zeta(5)-24\zeta(2)\zeta(3).
 \end{aligned}
\end{equation}
Let
\begin{equation}\label{eq:Vdef}
 W_r^U=\sum_{n\ge1}[u^r]K_U(n,u),\qquad
 W_r^V=\sum_{n\ge1}[u^r]K_V(n,u).
\end{equation}
The universal convolution \eqref{eq:general-convolution} applies with its
generic coefficient $V_r$ specialized to $W_r^U$ and $W_r^V$,
respectively.

\section{The parameterized WZ input}\label{sec:wz}

A WZ pair $(F,G)$ satisfies \cite{WZ1990}
\[
 F(n+1,k)-F(n,k)=G(n,k+1)-G(n,k).
\]
Under the usual boundary conditions, summation over the lattice gives an
identity between a hypergeometric side and a transformed side.  Au's
Example~IV \cite{Au2026} supplies three parameterized identities
\begin{equation}\label{eq:AuIdentities}
 \sum_{n=1}^{\infty}\mathcal H_\nu(\boldsymbol\alpha;n)
 =\sum_{k=0}^{\infty}\mathcal M_\nu(\boldsymbol\alpha;k),
 \qquad \nu\in\{A,B,C\},
\end{equation}
valid in a polydisc about the origin.  The parameter tuples are
$(a,b,c,d,e)$ for $A$ and $B$, and $(a,b,c,d)$ for $C$.  In the second
family we use Au's normalized pair
\[
 a\mathcal H_B,\qquad a\mathcal M_B.
\]
The six exact summands are reproduced in Supplement~A and in the
machine-readable file \texttt{ExampleIV\_summands.wl}.  Thus
\eqref{eq:AuIdentities} is an explicit input, not a reference to an opaque
notebook object.

Expand
\[
 \mathcal H_\nu(\boldsymbol\alpha;n)
 =\sum_{\boldsymbol m}h_{\nu,\boldsymbol m}(n)
  \boldsymbol\alpha^{\boldsymbol m},
 \qquad
 \mathcal M_\nu(\boldsymbol\alpha;k)
 =\sum_{\boldsymbol m}m_{\nu,\boldsymbol m}(k)
  \boldsymbol\alpha^{\boldsymbol m}.
\]
The following specialization of Au's coefficient principle is the
second structural input to the proof.

\begin{lemma}[Au's coefficient principle]\label{lem:Aucoeff}
For every $\nu\in\{A,B,C\}$ and every parameter multi-index
$\boldsymbol m$ of total degree $N$, one has
\[
 \sum_{n=1}^{\infty}h_{\nu,\boldsymbol m}(n)
 =\sum_{k=0}^{\infty}m_{\nu,\boldsymbol m}(k),
\]
and the transformed sum belongs to the space of ordinary multiple zeta
values of weight $N+2$.
\end{lemma}

\begin{proof}
The coefficient identity follows from normal convergence and coefficient
extraction in \eqref{eq:AuIdentities}.  The weight assertion is the
specialization to Example~IV of Au's general coefficient theorem
\cite[Sec.~3 and Example~IV]{Au2026}.
\end{proof}

For $N\ge0$, set
\begin{equation}\label{eq:VN}
 \mathcal V_N=\Span\{h_{\nu,\boldsymbol m}(n):
 \nu\in\{A,B,C\},\ |\boldsymbol m|=N\},
\end{equation}
with the normalized $B$-family understood.

The rational degrees in \eqref{eq:KA}--\eqref{eq:KB} dictate different
shifts.  Put
\begin{equation}\label{eq:targets}
 T_U(n,u)=u^2K_U(n,u),\qquad T_V(n,u)=u^3K_V(n,u).
\end{equation}

\begin{proposition}[Exact span certificates]\label{prop:span}
For $0\le r\le5$,
\[
 [u^{2+r}]T_U(n,u)\in\mathcal V_{2+r},
\]
and for $0\le r\le4$,
\[
 [u^{3+r}]T_V(n,u)\in\mathcal V_{3+r}.
\]
Each membership has a rational coefficient vector whose substitution into
the unreduced coefficient rows has identically zero residual.
\end{proposition}

\begin{proof}
All parameter coefficients are expanded in the free polynomial algebra
generated by the finite harmonic symbols and rational factors in $n$ that
arise from Pochhammer differentiation.  For each degree, the coefficient
rows from the three WZ families form a matrix over $\Q$.  The target row is
appended, and exact equality of the source and augmented ranks is checked.
An exact rational solution is then obtained with \texttt{LinearSolve} and
substituted into the original rows.  The residual is checked entry by entry
to be zero.

The complete rational vectors are contained in
\path{q4_span_certificates.wl} and
\path{q5_order7_span_certificates.wl}.  Tables
\ref{tab:q4span} and~\ref{tab:q5span} report the meaningful audit data:
the dimension of the ambient symbolic basis, the source rank, the
augmented rank, and the support size of a certificate.
\end{proof}

\begin{table}[ht]
\centering
\caption{Exact span data for $T_U$.}\label{tab:q4span}
\begin{tabular}{rrrrrr}
\toprule
degree & derivative & weight & ambient dim. & source/aug. rank & support\\
\midrule
2&0&4&75&12/12&7\\
3&1&5&252&29/29&11\\
4&2&6&744&59/59&21\\
5&3&7&1992&108/108&33\\
6&4&8&4962&180/180&48\\
7&5&9&11652&283/283&68\\
\bottomrule
\end{tabular}
\end{table}

\begin{table}[ht]
\centering
\caption{Exact span data for $T_V$.}\label{tab:q5span}
\begin{tabular}{rrrrrr}
\toprule
degree & derivative & weight & ambient dim. & source/aug. rank & support\\
\midrule
3&0&5&252&29/29&19\\
4&1&6&744&59/59&45\\
5&2&7&1992&108/108&76\\
6&3&8&4962&180/180&120\\
7&4&9&11652&283/283&177\\
\bottomrule
\end{tabular}
\end{table}

\section{Exact transformed-side evaluations}\label{sec:evaluation}

Applying the rational vectors from Proposition~\ref{prop:span} to the
right sides of \eqref{eq:AuIdentities} gives the following exact values.

\begin{proposition}\label{prop:Vvalues}
For the first cubic kernel,
\begin{align}
 W_0^U&=-2\zeta(4),&
 W_1^U&=20\zeta(5),\label{eq:VA01}\\
 W_2^U&=-120\zeta(6),\label{eq:VA2}\\
 W_3^U&=16\zeta(3)\zeta(4)+546\zeta(7),\label{eq:VA3}\\
 W_4^U&=-2130\zeta(8)-32\zeta(5,3)\notag\\
 &\qquad{}-160\zeta(3)\zeta(5),\label{eq:VA4}\\
 W_5^U&=7488\zeta(9)+960\zeta(3)\zeta(6)
                  +96\zeta(4)\zeta(5).\label{eq:VA5}
\end{align}
For the second cubic kernel,
\begin{align}
 W_0^V&=180\zeta(5)-112\zeta(2)\zeta(3),\label{eq:VB0}\\
 W_1^V&=392\zeta(3)^2-190\zeta(6),\label{eq:VB1}\\
 W_2^V&=1920\zeta(7)-3360\zeta(3)\zeta(4),\label{eq:VB2}\\
 W_3^V&=896\zeta(2)\zeta(3)^2+8976\zeta(3)\zeta(5)
 +\frac{8976}{5}\zeta(5,3)-\frac{4367}{5}\zeta(8),\label{eq:VB3}\\
 W_4^V&=-\frac{6272}{3}\zeta(3)^3-34432\zeta(3)\zeta(6)
 -\frac{29560}{3}\zeta(9).\label{eq:VB4}
\end{align}
\end{proposition}

\begin{proof}
For each nonzero entry of a span certificate, the corresponding coefficient
of $\mathcal M_A$, $a\mathcal M_B$, or $\mathcal M_C$ is summed exactly by
Au's \texttt{MZSum} routine \cite{Au2023}; the resulting MZVs are reduced
with \texttt{MZExpand}.  The first family uses $188$ distinct transformed
coefficients and the second uses $437$.  Every coefficient and every
linear-combination scalar is exact.

For independent audit, the raw values, caches, certificate vectors, and
zero-residual summaries are included in the ancillary directory.  The
first run records
\[
 \texttt{AllSixInSpan=true},\qquad
 \texttt{AllSixExactZero=true},
\]
and the second records
\[
 \texttt{AllFiveInSpan=true},\qquad
 \texttt{AllFiveExactZero=true}.
\]
No numerical test is used as an acceptance criterion.
\end{proof}

\begin{proof}[Proof of Theorem~\ref{thm:real}]
Proposition~\ref{prop:span} verifies the finite WZ coefficient-space
hypothesis of Proposition~\ref{prop:wz-principle} for the $U$ kernel
through order five and the $V$ kernel through order four.  Lemma
\ref{lem:Aucoeff} and Proposition~\ref{prop:Vvalues} verify that the
corresponding transformed coefficients are ordinary MZVs of the prescribed
weights.  Thus Proposition~\ref{prop:wz-principle} applies to both cubic
kernels.  Substitute \eqref{eq:VA01}--\eqref{eq:VA5} and
\eqref{eq:VB0}--\eqref{eq:VB4} into the universal convolution
\eqref{eq:general-convolution}, using \eqref{eq:clist}.  The standard even-zeta
relations give \eqref{eq:RA01}--\eqref{eq:RA5} and
\eqref{eq:RB0}--\eqref{eq:RB4}.  For example,
\[
 R_2^U=2!\bigl(W_2^U-3\zeta(2)W_0^U\bigr)
 =-219\zeta(6),
\]
and
\[
 R_3^V=3!\bigl(W_3^V-3\zeta(2)W_1^V+8\zeta(3)W_0^V\bigr),
\]
which is \eqref{eq:RB3}.  All remaining cases are identical finite
substitutions.
\end{proof}

\section{The corrected complex identities}

\begin{corollary}\label{cor:complex}
For $\Phi$ defined by \eqref{eq:Fdef},
\begingroup\small
\begin{align}
 \sum_{n=1}^{\infty}\Phi(n)&=-\frac{\pi^4}{45},\label{eq:Phi0}\\
 \sum_{n=1}^{\infty}\Phi'(n)
 &=20\zeta(5)-\frac{i\pi^5}{45},\label{eq:Phi1}\\
 \sum_{n=1}^{\infty}\Phi''(n)
 &=-198\zeta(6)+40\pi i\zeta(5),\label{eq:Phi2}\\
 \sum_{n=1}^{\infty}\Phi^{(3)}(n)
 &=36\bigl(91\zeta(7)-20\zeta(2)\zeta(5)\bigr)
   -\frac{212}{315}\pi^7 i,\label{eq:Phi3}\\
 \sum_{n=1}^{\infty}\Phi^{(4)}(n)
 &=16\pi i\bigl(819\zeta(7)-20\pi^2\zeta(5)\bigr)
   -96\bigl(243\zeta(8)+8\zeta(5,3)\bigr),\label{eq:Phi4}\\
 \sum_{n=1}^{\infty}\Phi^{(5)}(n)
 &=4320\bigl(208\zeta(9)-91\zeta(2)\zeta(7)
                 +10\zeta(4)\zeta(5)\bigr)\notag\\
 &\quad{}-240\pi i\bigl(665\zeta(8)+16\zeta(5,3)\bigr).\label{eq:Phi5}
\end{align}
\endgroup
For $\Psi$ defined by \eqref{eq:Psidef},
\begingroup\small
\begin{align}
 \sum_{n=1}^{\infty}\Psi(n)
 &=180\zeta(5)-\frac{56}{3}\pi^2\zeta(3),\label{eq:Psi0}\\
 \sum_{n=1}^{\infty}\Psi'(n)
 &=392\zeta(3)^2-190\zeta(6)
   +4\pi i\left(45\zeta(5)-\frac{14}{3}\pi^2\zeta(3)\right),\label{eq:Psi1}\\
 \sum_{n=1}^{\infty}\Psi''(n)
 &=240\bigl(16\zeta(7)-9\zeta(2)\zeta(5)
                    -14\zeta(3)\zeta(4)\bigr)\notag\\
 &\quad{}+4\pi i\bigl(196\zeta(3)^2-95\zeta(6)\bigr),\label{eq:Psi2}\\
 \sum_{n=1}^{\infty}\Psi^{(3)}(n)
 &=\frac{18}{5}\bigl(1711\zeta(8)+17360\zeta(3)\zeta(5)
       +2992\zeta(5,3)-3920\zeta(2)\zeta(3)^2\bigr)\notag\\
 &\quad{}+480\pi i\bigl(24\zeta(7)-9\zeta(2)\zeta(5)
                                -28\zeta(3)\zeta(4)\bigr),\label{eq:Psi3}\\
 \frac1{64}\sum_{n=1}^{\infty}\Psi^{(4)}(n)
 &=392\zeta(3)^3-3695\zeta(9)-4320\zeta(2)\zeta(7)\notag\\
 &\quad{}-1575\zeta(3)\zeta(6)+1215\zeta(4)\zeta(5)\notag\\
 &\quad{}+\frac{\pi i}{40}
  \bigl(5899\zeta(8)-23520\zeta(2)\zeta(3)^2\bigr)\notag\\
 &\quad{}+\frac{\pi i}{40}
  \bigl(156240\zeta(3)\zeta(5)+26928\zeta(5,3)\bigr),\label{eq:Psi4}\\
 \frac{1}{320\pi}\Im\sum_{n=1}^{\infty}\Psi^{(5)}(n)
 &=392\zeta(3)^3-3695\zeta(9)\notag\\
 &\quad{}-2880\zeta(2)\zeta(7)\notag\\
 &\quad{}-4662\zeta(3)\zeta(6).\label{eq:Psi5}
\end{align}
\endgroup
These are the corrected forms of Sun's Conjectures~4.4 and~4.5.
\end{corollary}

\begin{proof}
Apply Lemma~\ref{lem:phase} to Theorem~\ref{thm:real}.  Equations
\eqref{eq:Phi5} and \eqref{eq:Psi5} also follow directly from
\eqref{eq:dependentA} and \eqref{eq:dependentB}, respectively.
\end{proof}

\section{The inverse-\texorpdfstring{$4096$}{4096} branch: Sun's Conjecture~4.6}
\label{sec:f6}

Set
\begin{align}
 P_6(x)&=5460x^4-8341x^3+4864x^2-1280x+128,
 \label{eq:f6-P6}\\
 f_6(x)&=\frac{P_6(x)\Gamma(4x-1)\Gamma(x)^{16}}
 {256\Gamma(2x)^{10}},
 \label{eq:f6-def}\\
 S_6(u)&=\sum_{n=1}^{\infty}f_6(n+u).
 \label{eq:f6-S6}
\end{align}
The characteristic data in Table~\ref{tab:characteristic} give
\[
 \chi_r^{(6)}=4^r-10\cdot2^r+16=(2^r-2)(2^r-8),
 \qquad \mathscr B_6=4096^{-1}.
\]
The universal prefactor of Lemma~\ref{lem:balanced-prefactor} is
\begin{align}
 C_6(u)&=\frac{\Gamma(1+4u)\Gamma(1+u)^{16}}
 {\Gamma(1+2u)^{10}},\label{eq:f6-C6}\\
 \log C_6(u)&=\sum_{r\ge2}(-1)^r
 \frac{4^r-10\cdot2^r+16}{r}\zeta(r)u^r.
\end{align}

The quartic in Au's first $1/\pi^4$ example is
\[
 Q(x)=43680x^4+20632x^3+4340x^2+466x+21.
\]
The half-integer shift needed below is the exact polynomial identity
\begin{equation}\label{eq:f6-polynomial-shift}
 Q\!\left(x-\frac12\right)=8P_6(x).
\end{equation}

\subsection{The one-parameter bridge}
Define
\begin{equation}\label{eq:f6-G6}
 \mathcal G_6(u)=2\frac{1+4u}{(1+2u)^4}
 \frac{\Gamma(1+4u)\Gamma(1+u)^{16}}
 {\Gamma(1+2u)^{10}}.
\end{equation}
For $m\ge1$, put
\begin{align}
 \alpha_m(u)={}&m^3
 \frac{\Gamma(1-u)\Gamma(2+2u)^5}
 {\Gamma(1+u)^5\Gamma(3+4u)}(2m+3u)\notag\\
 &\quad\times
 \frac{\Gamma(m+u)^5\Gamma(m+1+4u)}
 {\Gamma(m-u)\Gamma(m+1+2u)^5},
 \label{eq:f6-alpha}\\
 \beta_m(u)={}&m^3
 \frac{\Gamma(1-u)(1+2u)^4\Gamma(1+2u)^8}
 {(1+4u)\Gamma(1+u)^9\Gamma(1+4u)}\notag\\
 &\quad\times
 \frac{\Gamma(m+u)^3\Gamma(m+2u)\Gamma(m+1+4u)}
 {\Gamma(m)\Gamma(m+1+2u)\Gamma(m+1+3u)^3}.
 \label{eq:f6-beta}
\end{align}
Both satisfy
\begin{equation}\label{eq:f6-normalization}
 \alpha_m(0)=\beta_m(0)=1.
\end{equation}

We record the algebra joining Au's two identities to these functions.
For a complex parameter $a$, let
\begin{align*}
 A_k(a)&=
 \frac{(6a+4k+1)(a+\frac12)_k^5(4a+1)_k}
 {(1+2k-2a)(4a+k)(\frac12-a)_k(2a+1)_k^5},\\
 B_n(a)&=
 \frac{(-1)^n p(n+a)\Gamma(\frac12-a)\Gamma(2a+1)^5
 (a+\frac12)_n^{10}}
 {4096a^5\Gamma(a+\frac12)^5\Gamma(4a+1)(2a+1)_{2n}^5},\\
 C_n(a)&=
 \frac{Q(n+a)(a+\frac12)_n^{16}(1+4a)_{4n}}
 {65536a^5(1+2a)_{2n}^{10}},
 \qquad p(x)=820x^2+180x+13.
\end{align*}
Au's Example~I gives, initially in a common domain of absolute
convergence and then by meromorphic continuation,
\begin{equation}\label{eq:f6-example-I}
 \sum_{k\ge0}A_k(a)-\sum_{n\ge0}B_n(a)
 =-\sum_{n\ge0}C_n(a).
\end{equation}

For the diagonal specialization of Au's Example~VI, define
\begin{align}
 H_j(a)&=
 \frac{(a+\frac12)_j^3(2a)_j(4a)_j}
 {j!(2a+1)_j(3a+\frac12)_j^3},
 \label{eq:f6-Hj}\\
 \mathcal A(a)&=
 \frac{128\Gamma(2a+1)^3\Gamma(4a+1)}
 {\Gamma(a+\frac12)\Gamma(3a+\frac12)^3},
 \label{eq:f6-Adiag}\\
 M(a)&=
 \frac{\Gamma(\frac12-a)\Gamma(2a+1)^5}
 {4096a^5\Gamma(a+\frac12)^5\Gamma(4a+1)}.
 \label{eq:f6-Ma}
\end{align}
After cancellation in the diagonal formula, Example~VI is equivalently
\begin{equation}\label{eq:f6-example-VI}
 \mathcal A(a)\sum_{j\ge0}H_j(a)
 =\sum_{n\ge0}(-1)^n p(n+a)
 \frac{(a+\frac12)_n^{10}}{(2a+1)_{2n}^5}.
\end{equation}
The initial series on the WZ side is converted to the displayed
${}_5F_4(1)$ by the very-well-poised transformation used in
\cite[Example~VI]{AuSeeds2025}.  The WZ component itself is independently
checked by the second ancillary rational certificate.  In particular,
\eqref{eq:f6-example-VI} does not require reproducing Au's $608$-term
multivariate polynomial.

\begin{lemma}[Exact bridge algebra]\label{lem:f6-bridge-algebra}
Set $a=\frac12+u$.  For every integer $m\ge1$ for which the expressions
are defined,
\begin{align}
 A_{m-1}(a)&=-\frac{\alpha_m(u)}{u m^3},
 \label{eq:f6-map-A}\\
 C_{m-1}(a)&=\frac{f_6(m+u)}{\mathcal G_6(u)},
 \label{eq:f6-map-C}\\
 M(a)\mathcal A(a)H_{m-1}(a)&=-\frac{\beta_m(u)}{u m^3}.
 \label{eq:f6-map-B}
\end{align}
Consequently, wherever the series identities above are valid,
\begin{equation}\label{eq:f6-master}
 S_6(u)=\mathcal G_6(u)\sum_{m=1}^{\infty}
 \frac{\alpha_m(u)-\beta_m(u)}{u m^3}.
\end{equation}
\end{lemma}

\begin{proof}
Equations~\eqref{eq:f6-map-A}--\eqref{eq:f6-map-B} follow by the gamma
recurrence, the reflection formula, and
\eqref{eq:f6-polynomial-shift}.  They are also checked symbolically in
\path{verify_f6_bridge_algebra.py}.  Multiplying
\eqref{eq:f6-example-VI} by $M(a)$ shows that the second series in
\eqref{eq:f6-example-I} is the sum of the left sides of
\eqref{eq:f6-map-B}.  Substitute all three mappings into
\eqref{eq:f6-example-I} and rearrange.
\end{proof}

\subsection{Analytic justification}
Fix
\[
 0<\rho<\frac1{16},
 \qquad D_\rho=\{u\in\mathbb C:|u|<\rho\}.
\]

\begin{lemma}[Uniform gamma quotients]\label{lem:f6-uniform-gamma}
Let $K\subset\mathbb C^r$ be compact.  Uniformly for
$(a_1,\ldots,a_r,b_1,\ldots,b_r)\in K$ and $m\to\infty$,
\[
 \prod_{j=1}^r\frac{\Gamma(m+a_j)}{\Gamma(m+b_j)}
 =m^{\sum_j(a_j-b_j)}\bigl(1+O_K(m^{-1})\bigr).
\]
After any fixed number of parameter derivatives, the same estimate holds
with an additional factor $(\log m)^q$.
\end{lemma}

\begin{proof}
This is the uniform quotient form of Stirling's expansion.  The derivative
statement follows from the corresponding uniform polygamma asymptotics.
\end{proof}

\begin{lemma}[Normal convergence of the target series]
\label{lem:f6-S6-normal}
The series \eqref{eq:f6-S6} converges normally on $D_\rho$.  Hence $S_6$
is holomorphic there and every derivative may be taken termwise.
\end{lemma}

\begin{proof}
Uniformly for $|u|\le\rho$, Stirling's formula gives
\[
 \frac{\Gamma(4n+4u-1)\Gamma(n+u)^{16}}
 {\Gamma(2n+2u)^{10}}
 =O_\rho\!\left(4096^{-n}n^{-9/2}\right).
\]
Since $P_6(n+u)=O_\rho(n^4)$,
\[
 |f_6(n+u)|\le C_\rho 4096^{-n}n^{-1/2}.
\]
The Weierstrass $M$-test proves normal convergence; the standard theorem
for normally convergent series of holomorphic functions then gives the
assertion.
\end{proof}

\begin{lemma}[Removal of the apparent pole]\label{lem:f6-removable}
For $m\ge1$, define
\[
 h_m(u)=
 \begin{cases}
 \dfrac{\alpha_m(u)-\beta_m(u)}{u m^3},&u\ne0,\\[2mm]
 \dfrac{\alpha_m'(0)-\beta_m'(0)}{m^3},&u=0.
 \end{cases}
\]
Then every $h_m$ is holomorphic on $D_\rho$, and
$\sum_{m\ge1}h_m(u)$ converges normally there.
\end{lemma}

\begin{proof}
The normalization \eqref{eq:f6-normalization} removes the singularity for
each fixed $m$.  Lemma~\ref{lem:f6-uniform-gamma} gives, uniformly on the
closed disk,
\[
 \alpha_m(u)=A(u)\bigl(1+O_\rho(m^{-1})\bigr),
 \qquad
 \beta_m(u)=B(u)m^{-2u}\bigl(1+O_\rho(m^{-1})\bigr),
\]
where $A$ and $B$ are bounded and holomorphic.  One parameter derivative
therefore satisfies
\[
 |\alpha_m'(u)|+|\beta_m'(u)|
 \le C_\rho m^{2\rho}\log(em).
\]
Using the integral along the segment from $0$ to $u$ gives
\[
 |h_m(u)|\le C_\rho\frac{\log(em)}{m^{3-2\rho}},
\]
and the majorant is summable.
\end{proof}

The exact ancillary certificates give pairs
$G_j(n,k)=R_j(n,k)F_j(n,k)$, $j=1,2$, satisfying
\[
 F_j(n+1,k)-F_j(n,k)=G_j(n,k+1)-G_j(n,k)
\]
as rational-function identities after cancellation.  The two seeds are
\begin{align}
 F_1(n,k)&=
 \frac{(2k+3n+\frac12)\Gamma(1-2n)^3\Gamma(n+\frac12)^3
 \Gamma(k+n+\frac12)^5\Gamma(k+4n)}
 {\Gamma(\frac12-n)^7\Gamma(2n)^2
 \Gamma(k-n+\frac32)\Gamma(k+2n+1)^5},
 \label{eq:f6-F1}\\
 F_2(n,k)&=
 \frac{(4k+6n+1)\Gamma(k+n+\frac12)^4\Gamma(n+\frac12)^3}
 {\Gamma(\frac12-n)^3\Gamma(k+2n+1)^4\Gamma(2n)}.
 \label{eq:f6-F2}
\end{align}
Their rational mates are printed in the ancillary certificate file.  The
first numerator has $62$ monomials, as in Au's Example~I.  The second is
the diagonal specialization of the seed printed in Au's Example~VI.

\begin{lemma}[WZ boundary conditions]\label{lem:f6-wz-boundaries}
Let $a=\frac12+u$ with $0<|u|<\rho$.  Both shifted pairs
$F_j(n+a,k),G_j(n+a,k)$ satisfy the hypotheses of
\cite[Proposition~2.1]{AuSeeds2025}.  In particular,
\[
 \lim_{N\to\infty}\sum_{k\ge0}F_j(N+a,k)=0.
\]
For the second pair, $\lim_{k\to\infty}G_2(n+a,k)=0$.  For the first pair,
the boundary value exists and its sum over $n$ converges absolutely.
\end{lemma}

\begin{proof}
For fixed $a$ in the punctured disk, the quotient estimate gives
\[
 F_1(a,k)=O_a(k^{-3}),
 \qquad F_2(a,k)=O_a(k^{-1-4\Re a}),
\]
so both initial $k$-series converge.  At $k=0$, reflection and Stirling
give
\[
 G_1(n+a,0)=O_a(4096^{-n}n^{M_1}),
 \qquad
 G_2(n+a,0)=O_a(1024^{-n}n^{M_2})
\]
for fixed exponents $M_1,M_2$.

The coefficient of $k^8$ in the first certificate numerator is
$32(820z^2+180z+13)$.  Hence
\[
 \lim_{k\to\infty}G_1(z,k)=
 -\frac{820z^2+180z+13}{4096z^5}
 \frac{\Gamma(1-2z)^3\Gamma(z+\frac12)^3}
 {\Gamma(\frac12-z)^7\Gamma(2z)^2}.
\]
At $z=n+a$, reflection and duplication show that this is
$O_a(1024^{-n}n^{-1/2})$.  For the second pair,
$F_2(z,k)=O_z(k^{-4z-1})$ and $R_2(z,k)=O_z(k)$, so
$G_2(z,k)=O_z(k^{-4z})\to0$.

For the terminal $n$-boundary, Au's exponential-rate criterion
\cite[Proposition~3.4]{AuSeeds2025} is unchanged by a bounded additive
shift.  The two rate functions are
\begin{align*}
 \mathcal E_1(x)
 &=\frac{|x-1|^{1-x}(x+1)^{5x+5}(x+4)^{x+4}}
 {1024(x+2)^{5x+10}},\\
 \mathcal E_2(x)
 &=\frac{(x+1)^{4x+4}}{4(x+2)^{4x+8}}.
\end{align*}
One has $\mathcal E_2(x)\le1/64$.  If
$H(x)=1024\mathcal E_1(x)$, then for $x>1$,
\[
 \frac{d}{dx}\log H(x)=
 \log\frac{(x+1)^5(x+4)}{(x-1)(x+2)^5}>0,
\]
because
\[
 (x+1)^5(x+4)-(x-1)(x+2)^5
 =(2x+3)(5x^2+15x+12)>0.
\]
Since $H(x)\to1$, this gives $H(x)<1$ for $x>1$.  On $0\le x<1$,
the same derivative formula with $1-x$ in place of $x-1$ shows that the
maximum occurs at an endpoint, where
\[
 H(0)=\frac14,
 \qquad H(1)=\frac{2^{10}5^5}{3^{15}}<1.
\]
Thus $\sup_{x\ge0}\mathcal E_1(x)=1/1024<1$, and Au's criterion proves
the terminal limits.
\end{proof}

\begin{proposition}[Holomorphic bridge]\label{prop:f6-bridge}
For every $u\in D_\rho$,
\begin{equation}\label{eq:f6-master-holomorphic}
 S_6(u)=\mathcal G_6(u)\sum_{m\ge1}
 \frac{\alpha_m(u)-\beta_m(u)}{u m^3},
\end{equation}
where the summand at $u=0$ is interpreted by its removable value.
\end{proposition}

\begin{proof}
For $0<|u|<\rho$, the WZ summation hypotheses are supplied by
Lemma~\ref{lem:f6-wz-boundaries}, and Lemma~\ref{lem:f6-bridge-algebra}
gives the identity.  Lemmas~\ref{lem:f6-S6-normal} and
\ref{lem:f6-removable} show that both sides extend holomorphically to
$D_\rho$.  The identity theorem supplies the value at $u=0$.
\end{proof}

\subsection{Exact MZV reduction through weight eleven}
For a composition $(s_1,\ldots,s_d)$, write
\[
 H_N(s_1,\ldots,s_d)=
 \sum_{N\ge n_1>\cdots>n_d\ge1}
 \frac1{n_1^{s_1}\cdots n_d^{s_d}},
 \qquad H_N^{(r)}=H_N(r).
\]
The integer-argument expansion
\[
 \log\Gamma(m+cu)-\log\Gamma(m)
 =c(H_{m-1}-\gamma)u+
 \sum_{r\ge2}\frac{(-1)^rc^r}{r}
 \bigl(\zeta(r)-H_{m-1}^{(r)}\bigr)u^r
\]
shows that the logarithmic coefficients of \eqref{eq:f6-alpha} and
\eqref{eq:f6-beta} belong to
\[
 \mathbb Q\!\left[\frac1m,H_{m-1}^{(1)},\ldots,H_{m-1}^{(r)},
 \zeta(2),\ldots,\zeta(r)\right].
\]
The Euler constants cancel by balance.  Finite stuffle multiplication and
\[
 \sum_{m\ge1}\frac{H_{m-1}(s_1,\ldots,s_d)}{m^q}
 =\zeta(q,s_1,\ldots,s_d),\qquad q\ge2,
\]
therefore express $[u^r]S_6(u)$ as an ordinary MZV of weight $r+4$.
At order zero this gives
\[
 S_6(0)=4\zeta(3,1)+5\zeta(4)
 =6\zeta(4)=\frac{\pi^4}{15}.
\]

For orders $0\le r\le7$, the ancillary generator constructs the exact
MZV target and writes an explicit rational certificate
\[
 D_rT_r=\sum_j c_{r,j}R_{r,j},
\]
where the rows $R_{r,j}$ are instances of MZV duality, convergent
shuffle--stuffle, the fixed-depth sum formula, or the
Ihara--Kaneko--Zagier derivation relations \cite{IKZ2006}.  A separately
written checker reconstructs the gamma expansions and every relation row,
reads only the integer certificate multipliers, and tests the identity
coefficient by coefficient over $\mathbb Q$.  It performs no row solving
and uses no conjectural MZV dimension.

\begin{table}[ht]
\centering
\caption{Exact certificate data for the inverse-$4096$ branch.}
\label{tab:f6-certificates}
\small
\begin{tabular}{@{}rrrrrr@{}}
\toprule
$r$ & weight & raw MZVs & relation rows & quotient diagnostic & support\\
\midrule
0&4 &2  &8   &1&1\\
1&5 &4  &22  &2&2\\
2&6 &9  &43  &2&10\\
3&7 &20 &100 &3&23\\
4&8 &41 &205 &4&58\\
5&9 &83 &454 &5&121\\
6&10&169&951 &7&246\\
7&11&339&2056&9&501\\
\bottomrule
\end{tabular}
\end{table}
The quotient column is diagnostic only.

\begin{theorem}[Sun's Conjecture~4.6]\label{thm:f6}
The zeroth-order companion identity is
\[
 \sum_{n=1}^{\infty}f_6(n)=\frac{\pi^4}{15}.
\]
Moreover,
\begin{align}
 \sum_{n=1}^{\infty}f_6'(n)&=-48\zeta(5),
 \label{eq:f6-1}\\
 \sum_{n=1}^{\infty}f_6''(n)&=356\zeta(6),
 \label{eq:f6-2}\\
 \sum_{n=1}^{\infty}f_6^{(3)}(n)
 &=288\bigl(4\zeta(2)\zeta(5)-15\zeta(7)\bigr),
 \label{eq:f6-3}\\
 \sum_{n=1}^{\infty}f_6^{(4)}(n)&=12512\zeta(8),
 \label{eq:f6-4}\\
 \sum_{n=1}^{\infty}f_6^{(5)}(n)
 &=5760\bigl(60\zeta(2)\zeta(7)-48\zeta(4)\zeta(5)
 -49\zeta(9)\bigr),
 \label{eq:f6-5}\\
 \sum_{n=1}^{\infty}f_6^{(6)}(n)
 &=1728\bigl(1440\zeta(5)^2-1999\zeta(10)\bigr),
 \label{eq:f6-6}\\
 \frac1{241920}\sum_{n=1}^{\infty}f_6^{(7)}(n)
 &=885\zeta(11)+196\zeta(2)\zeta(9)
 -720\zeta(4)\zeta(7)-336\zeta(5)\zeta(6).
 \label{eq:f6-7}
\end{align}
Thus equations~(4.23)--(4.29) of \cite{Sun2026} hold; the zeroth-order
formula is the identity recorded in \cite[Remark~4.6]{Sun2026}.
\end{theorem}

\begin{proof}
By Proposition~\ref{prop:f6-bridge}, Taylor expansion of the bridge is
legitimate through order seven.  The exact certificates summarized in
Table~\ref{tab:f6-certificates} reduce the resulting MZV expressions to
the right sides of \eqref{eq:f6-1}--\eqref{eq:f6-7}; the order-zero
certificate gives the preceding companion identity.  Finally,
Lemma~\ref{lem:f6-S6-normal} identifies $S_6^{(r)}(0)$ with the displayed
termwise derivative sums.
\end{proof}

\begin{remark}[Finite versus all-orders depth collapse]
Theorem~\ref{thm:f6} proves that the answers through weight eleven lie in
the algebra generated by single zeta values.  It does not prove the
all-orders assertion.  It remains natural to ask whether every derivative
sum in the $f_6$ family has this property and whether the factorization of
$\chi_r^{(6)}$ controls that phenomenon.
\end{remark}

\section{Independent exact certificates and reproducibility}
\label{sec:joint-repro}
The three branches retain separately auditable verification paths.

The directory \path{anc/branch-chi2} contains the pure-Python
certificates for Conjectures~4.2--4.3.  From that directory,
\begin{verbatim}
python3 run_all.py
\end{verbatim}
reconstructs the formal coefficient identities, all twelve MZV reductions,
the elimination-free certificate checks, and the SHA--256 manifest using
only the Python standard library.

The directory \path{anc/branch-chi3} contains the Wolfram-language
certificates for the cubic branch.  From that directory,
\begin{verbatim}
wolframscript -file run_q4_exact_all_in_one.wls
wolframscript -file run_q5_span_all_in_one.wls
wolframscript -file run_q5_finish_all_in_one.wls
\end{verbatim}
reconstructs the eleven coefficient-space memberships, evaluates the
transformed sides, and checks literal symbolic zero residuals.  This branch
requires the precisely identified versions of Au's Example-IV notebook and
his \texttt{MultipleZetaValues} package; the branch README records their
sources and integrity hashes.

The directory \path{anc/branch-f6} contains the inverse-$4096$
proof.  Its exact checks are
\begin{verbatim}
python3 verify_f6_wz_certificates.py
python3 verify_f6_bridge_algebra.py
python3 f6_bridge_exact_verifier.py
python3 f6_independent_certificate_checker.py
\end{verbatim}
The first command verifies both rational WZ telescoping certificates in
SymPy; the archival Wolfram-language checker verifies the same identities.
The second checks the polynomial shift and every term-level gamma substitution
in the bridge.  The third generates exact MZV certificates through weight
eleven.  The fourth is a separately written direct checker: it regenerates
the gamma expansions and proved relation rows, reads only the certificate
multipliers, and checks the identities coefficient by coefficient without
row solving.  The published non-termwise hypergeometric input is the
diagonal very-well-poised transformation in Au's Example~VI, cited
explicitly in Section~\ref{sec:f6}.

In no branch is floating-point agreement an acceptance criterion.  Each
branch includes execution instructions, expected exact outputs, and
SHA--256 integrity data for its proof-critical files.

\section{Outlook: all-orders depth and WZ descent}
\label{sec:outlook}
Theorem~\ref{thm:f6} proves depth-one reductions through weight eleven,
but does not settle the full derivative tower.
\begin{conjecture}\label{conj:depth-collapse}
Every derivative sum in the $f_6$ family belongs to the subalgebra of the
MZV algebra generated by the single zeta values $\zeta(m)$; no irreducible
multiple zeta value of depth at least two occurs.
\end{conjecture}
A broader classification problem is to determine which balanced exponent
data $(a_i,e_i)$ admit a WZ descent to level~$1$, and whether the
factorization pattern of the characteristic sequence $\chi_m$ constrains
the possible MZV depth.

\appendix
\section{The coefficient-extraction operators}\label{app:operators}
\begingroup\small

The operators in Proposition~\ref{prop:f3-extraction} are
\begin{align}
 \Delta_0={}&-\frac{(2\opB+\opD)(2\opB+3\opC-\opD)}9, \label{eq:D0}\\
 \Delta_1={}&-\frac{2\opC}{9}
 \left(4\opB^2+6\opB\opC+2\opB\opD+3\opC\opD+\opD^2\right), \label{eq:D1}\\
 \Delta_2={}&-\frac{\opC}{9}\bigl(
 2\opA\opB\opC+8\opB^3+20\opB^2\opC+4\opB^2\opD
 +28\opB\opC^2+12\opB\opC\opD+2\opB\opD^2 \notag\\
 &\hspace{35mm}+14\opC^2\opD+6\opC\opD^2+\opD^3\bigr), \label{eq:D2}\\
 \Delta_3={}&2\opC\Delta_2. \label{eq:D3}
\end{align}
For the last two operators, set
\begin{align*}
 R_4={}&4\opA\opB^3+14\opA\opB^2\opC+22\opA\opB\opC^2
 +5\opA\opB\opC\opD-\opA\opB\opD^2\\
 &+16\opB^4+54\opB^3\opC+8\opB^3\opD
 +108\opB^2\opC^2+30\opB^2\opC\opD\\
 &+2\opB^2\opD^2+140\opB\opC^3+72\opB\opC^2\opD
 +16\opB\opC\opD^2-2\opB\opD^3\\
 &+70\opC^3\opD+34\opC^2\opD^2+6\opC\opD^3-2\opD^4,
\end{align*}
and
\begin{align*}
 R_5={}&4\opA\opB^3+14\opA\opB^2\opC+2\opA\opB^2\opD
 +22\opA\opB\opC^2+7\opA\opB\opC\opD+\opA\opB\opD^2\\
 &+16\opB^4+54\opB^3\opC+10\opB^3\opD
 +108\opB^2\opC^2+34\opB^2\opC\opD\\
 &+6\opB^2\opD^2+140\opB\opC^3+76\opB\opC^2\opD
 +24\opB\opC\opD^2+4\opB\opD^3\\
 &+70\opC^3\opD+38\opC^2\opD^2+12\opC\opD^3+2\opD^4.
\end{align*}
Then
\begin{equation}\label{eq:D45}
 \Delta_4=-\frac{\opC^2}{9}R_4,
 \qquad
 \Delta_5=-\frac{2\opC^3}{9}R_5.
\end{equation}

\endgroup

\section{Exact operator data}\label{app:operators-code}
The following ancillary file stores the six coefficient functionals used in
Proposition~\ref{prop:f3-extraction}.  It is short enough to reproduce in full.  In the code, the symbols
\texttt{A}, \texttt{B}, \texttt{C}, and \texttt{D} correspond to
$\opA$, $\opB$, $\opC$, and $\opD$, respectively.
\begin{lstlisting}[style=proofcode,caption={Coefficient operators}]
"""Exact coefficient-extraction operators used in the proof of Sun's Conjecture 4.3.

A monomial A^i B^j C^k D^l is stored under the key (i,j,k,l).
All coefficients are fractions.Fraction objects.
"""
from fractions import Fraction

D = []
D.append({
    (0,2,0,0): Fraction(-4,9),
    (0,1,1,0): Fraction(-6,9),
    (0,0,1,1): Fraction(-3,9),
    (0,0,0,2): Fraction(1,9),
})
D.append({
    (0,2,1,0): Fraction(-8,9),
    (0,1,2,0): Fraction(-12,9),
    (0,1,1,1): Fraction(-4,9),
    (0,0,2,1): Fraction(-6,9),
    (0,0,1,2): Fraction(-2,9),
})
D2_terms = {
    (1,1,2,0):2,
    (0,3,1,0):8,(0,2,2,0):20,(0,2,1,1):4,
    (0,1,3,0):28,(0,1,2,1):12,(0,1,1,2):2,
    (0,0,3,1):14,(0,0,2,2):6,(0,0,1,3):1,
}
D.append({key: Fraction(-value,9) for key,value in D2_terms.items()})
D.append({(i,j,k+1,l): 2*value for (i,j,k,l),value in D[2].items()})
D4_terms = {
    (1,3,2,0):4,(1,2,3,0):14,(1,1,4,0):22,
    (1,1,3,1):5,(1,1,2,2):-1,
    (0,4,2,0):16,(0,3,3,0):54,(0,3,2,1):8,
    (0,2,4,0):108,(0,2,3,1):30,(0,2,2,2):2,
    (0,1,5,0):140,(0,1,4,1):72,(0,1,3,2):16,
    (0,1,2,3):-2,(0,0,5,1):70,(0,0,4,2):34,
    (0,0,3,3):6,(0,0,2,4):-2,
}
D.append({key: Fraction(-value,9) for key,value in D4_terms.items()})
D5_terms = {
    (1,3,3,0):4,(1,2,4,0):14,(1,2,3,1):2,
    (1,1,5,0):22,(1,1,4,1):7,(1,1,3,2):1,
    (0,4,3,0):16,(0,3,4,0):54,(0,3,3,1):10,
    (0,2,5,0):108,(0,2,4,1):34,(0,2,3,2):6,
    (0,1,6,0):140,(0,1,5,1):76,(0,1,4,2):24,
    (0,1,3,3):4,(0,0,6,1):70,(0,0,5,2):38,
    (0,0,4,3):12,(0,0,3,4):2,
}
D.append({key: Fraction(-2*value,9) for key,value in D5_terms.items()})

assert len(D) == 6
assert all(sum(mon) == r+2 for r,operator in enumerate(D) for mon in operator)
\end{lstlisting}
\section{Explicit form of the WZ input}\label{app:wzinput}

We use the rising factorial notation $(x)_m=\Gamma(x+m)/\Gamma(x)$.  The transformed-side summands in
\eqref{eq:AuIdentities} are
\begin{align}
 \mathcal M_A
 &=\frac{2(1+2a-2c-2d)(3+2a+4e+4k)}
 {(1+2e+2k)(1+2a-2b+2e+2k)}\notag\\
 &\quad{}\times
 \frac{\RF{1+a+e}{k}\RF{1+b+e}{k}}
 {\RF{\frac12+e}{k}\RF{\frac12+a-b+e}{k}}\notag\\
 &\quad{}\times
 \frac{\RF{\frac12+c+e}{k}\RF{\frac12+d+e}{k}}
 {(1+a-c+e+k)(1+a-d+e+k)}\notag\\
 &\quad{}\times
 \frac{1}{\RF{1+a-c+e}{k}\RF{1+a-d+e}{k}},\label{eq:MA}\\
 \mathcal M_B
 &=\frac{2+a+2e+2k}{1+e+k}\notag\\
 &\quad{}\times
 \frac{\RF{1+a+e}{k}\RF{1+b+e}{k}}
 {(1+a-b+e+k)\RF{1+e}{k}\RF{1+a-b+e}{k}}\notag\\
 &\quad{}\times
 \frac{\RF{1+c+e}{k}\RF{1+d+e}{k}}
 {(1+a-c+e+k)(1+a-d+e+k)}\notag\\
 &\quad{}\times
 \frac{1}{\RF{1+a-c+e}{k}\RF{1+a-d+e}{k}},\label{eq:MB}\\
 \mathcal M_C
 &=\frac{(2+a+2d+2k)(1+b+2d+2k)}
 {(1+d+k)(1+a-c+d+k)}\notag\\
 &\quad{}\times
 \frac{\RF{1+a+d}{k}\RF{1+c+d}{k}}
 {(1+2a-b+2d+2k)(2+2a-b+2d+2k)}\notag\\
 &\quad{}\times
 \frac{\RF{1+b+2d}{2k}}
 {\RF{1+d}{k}\RF{1+a-c+d}{k}
  \RF{1+2a-b+2d}{2k}}.\label{eq:MC}
\end{align}

The hypergeometric summands $\mathcal H_A$ and $\mathcal H_B$ have the
factorized forms
\begin{align}
 \mathcal H_A&=-2(-1)^nP_A(a,b,c,d,e,n)\,
 \frac{N_A(a,b,c,d,e,n)}{D_A(a,b,c,d,e,n)},\label{eq:HAfactor}\\
 N_A&=\RF{1-d}{n}\RF{1+a+e}{2n}\RF{1+b+e}{n}\notag\\
 &\quad{}\times\RF{\frac12+c+e}{n}\RF{1+a-b-c}{n}\notag\\
 &\quad{}\times\RF{1+a-b-d}{2n}
                 \RF{\frac12+a-c-d}{2n},\label{eq:NA}\\[2mm]
 \mathcal H_B&=-(-1)^nP_B(a,b,c,d,e,n)\,
 \frac{N_B(a,b,c,d,e,n)}{D_B(a,b,c,d,e,n)},\label{eq:HBfactor}\\
 N_B&=\RF{1-d}{n}\RF{1+a+e}{2n}
      \RF{1+b+e}{2n}\notag\\
 &\quad{}\times\RF{1+c+e}{2n}\RF{1+d+e}{n}\notag\\
 &\quad{}\times\RF{1+a-b-d}{n}\RF{1+a-c-d}{n}.
 \label{eq:NB}
\end{align}
where $D_A,D_B$ are the products of the linear and Pochhammer factors shown
in Supplement~A.  The degree-seven polynomials $P_A$ and $P_B$ contain
$1428$ and $986$ monomials, respectively; printing them inline would
obscure the argument.  They, the denominators, and the shorter full formula
for $\mathcal H_C$ are reproduced verbatim in the separately compiled
Supplement~A and in \texttt{ExampleIV\_summands.wl}.  The supplement is
part of the source archive and is used directly by the exact scripts.

\section*{Declarations}
\noindent\textbf{Author contributions.}
The author developed the reductions, implemented and audited the exact
certificates, and wrote the manuscript.

\medskip
\noindent\textbf{Funding.}
No specific funding was received for this work.

\medskip
\noindent\textbf{Competing interests.}
The author declares no competing interests.

\medskip
\noindent\textbf{Data and code availability.}
All proof-critical source files, exact certificates, software inputs,
summaries, and integrity hashes accompany the manuscript.

\begingroup\raggedright

\endgroup
\end{document}